\documentclass[11pt, a4paper, oneside, fleqn]{article}
\usepackage[british]{babel}
\usepackage[english]{varioref}
\usepackage{amsmath,amsthm,amssymb}
\usepackage{enumitem}
\usepackage[square, sort, numbers]{natbib}
\usepackage{lscape}
\usepackage{multicol}
\usepackage{float}
\usepackage{multirow}
\usepackage{algorithmic}
\usepackage[plain]{algorithm}
\usepackage{longtable}
\usepackage{comment}
\usepackage{url}
\usepackage{hyperref}
\usepackage{tikz}

\newcommand*{\FULLVERSION}{}

\includecomment{comment:instanceREA150MainIdeaOfOurApproachIterationAll}
\includecomment{comment:computationalResultsForRandomEuclideanInstances}

\def\mainApproach{BasicIntegerTSP}

\floatstyle{ruled}

\newcommand{\argmin}[1]{\underset{#1}{\operatorname{arg}\,\operatorname{min}}\;}

\newcommand*{\defeq}{\mathrel{\vcenter{\baselineskip0.5ex \lineskiplimit0pt
                     \hbox{\scriptsize.}\hbox{\scriptsize.}}}%
                     =}

\makeatletter
	\newenvironment{figurehere}
		{\def\@captype{figure}}
		{}
\makeatother

\usetikzlibrary{decorations.pathreplacing,decorations.pathmorphing}

\begin{document}
\newtheorem{axiom}{Axiom}
\newtheorem{conjecture}[axiom]{Conjecture}
\newtheorem{corollary}[axiom]{Corollary}
\newtheorem{definition}[axiom]{Definition}
\newtheorem{example}[axiom]{Example}
\newtheorem{lemma}[axiom]{Lemma}
\newtheorem{observation}[axiom]{Observation}
\newtheorem{proposition}[axiom]{Proposition}
\newtheorem{theorem}[axiom]{Theorem}

\newcommand{\rz}{{\mathbb{R}}}
\newcommand{\nz}{{\mathbb{N}}}
\newcommand{\zz}{{\mathbb{Z}}}
\newcommand{\eps}{\varepsilon}
\newcommand{\cei}[1]{\lceil #1\rceil}
\newcommand{\flo}[1]{\left\lfloor #1\right\rfloor}
\newcommand{\seq}[1]{\langle #1\rangle}
\newcommand{\qap}{\mbox{\sc QAP}}
\newcommand{\minqap}{\mbox{$\min$-{\sc QAP}}}
\newcommand{\maxqap}{\mbox{$\max$-{\sc QAP}}}
\newcommand{\onion}{\mbox{\sc Onion}}
\newcommand{\ocone}{\mbox{\sc OnionCone}}
\newcommand{\aaa}{\alpha}
\newcommand{\minctv}{\mbox{$\min$-{\sc CTV}}}
\newcommand{\maxctv}{\mbox{$\max$-{\sc CTV}}}
\newcommand{\cbar}{\overline{C}{}}
\newcommand{\var}{\mbox{\sc Var}}

\title{\bf Generating subtour elimination constraints \\ for the TSP from pure integer solutions}
\author{
	\sc Ulrich Pferschy{\footnotemark[1]}
	\and
	\sc Rostislav Stan\v{e}k{\footnotemark[1]}
}
\date{}
\maketitle
\renewcommand{\thefootnote}{\fnsymbol{footnote}}
\footnotetext[1]{
	{\tt \{pferschy, rostislav.stanek\}@uni-graz.at}.
	Department of Statistics and Operations Research, University of Graz, 
	Universit\"atsstra{\ss}e 15, 8010 Graz, Austria}
\renewcommand{\thefootnote}{\arabic{footnote}}

\begin{abstract}
The {\em traveling salesman problem} ({\em TSP}) is one of the most prominent combinatorial optimization problems. Given a complete graph $G = (V, E)$ and non-negative distances d for every edge, the TSP asks for a shortest tour through all vertices with respect to the distances d. The method of choice for solving the TSP to optimality is a {\em branch and cut approach}. Usually the {\em integrality constraints} are relaxed first and all separation processes to identify violated inequalities are done on {\em fractional solutions}.

In our approach we try to exploit the impressive performance of current ILP-solvers and work only with integer solutions without ever interfering with fractional solutions. We stick to a very simple ILP-model and relax the {\em subtour elimination constraints} only. The resulting problem is solved to integer optimality, violated constraints (which are trivial to find) are added and the process is repeated until a feasible solution is found.
	
In order to speed up the algorithm we pursue several attempts to find as many {\em relevant} subtours as possible. These attempts are based on the clustering of vertices with additional insights gained from empirical observations and random graph theory. Computational results are performed on test instances taken from the \mbox{{\em TSPLIB95}} and on {\em random Euclidean graphs}.
\end{abstract}



\section{Introduction}
\label{sec:intro}

The {\em Traveling Salesman/Salesperson Problem} {\em TSP} is one of the best known
and most widely investigated combinatorial optimization problems
with four famous books entirely devoted to its study
(\cite{llr85}, \cite{TheTravelingSalesmanComputationalSolutionsForTSPApplications}, \cite{gupu06}, \cite{TheTravelingSalesmanProblemAComputionalStudy}).
Thus, we will refrain from giving extensive references
but mainly refer to the treatment in \cite{TheTravelingSalesmanProblemAComputionalStudy}.
Given a complete graph $G = (V, E)$ with $|V| = n$ and
$|E| = m=n(n-1)/2$, and nonnegative distances $d_e$ for each $e \in E$,
the TSP asks for a shortest tour with respect to the distances $d_e$
containing each vertex exactly once.

Let $\delta(v):=\{e=(v,u)\in E \mid u \in V\}$ denote the set of all edges
adjacent to $v\in V$.
Introducing binary variables $x_e$ for the possible inclusion of any edge $e\in E$ in the tour
we get the following classical ILP formulation:
	\begin{alignat}{5}
		\label{equation:TSP:ILP1}
		\mbox{minimize} \quad
			&& \sum_{e \in E} d_e x_e \quad
				&&
					&&
						&&\\
		\label{equation:TSP:ILP2}
		\mbox{s.t.} \quad
			&& \sum_{e \in \delta(v)}{x_e} \quad
				&& = \quad
					&&& 2 \quad
						&& \forall\; v \in V,\\
		\label{equation:TSP:ILP3}
			&& \sum_{\substack{{e=(u,v) \in E}\\{u,v \in S}}}{x_e} \quad
				&& \leq \quad
					&&& |S| - 1 \quad
						&&  \forall\;S \subset V, S \neq\emptyset,\\
		\label{equation:TSP:ILP4}
			&& x_e \quad
				&& \in \quad
					&&& \{0, 1\} \quad
						&& \forall\; e \in E
	\end{alignat}
(\ref{equation:TSP:ILP1}) defines the {\em objective function},
	(\ref{equation:TSP:ILP2}) is the {\em degree equation} for each vertex,
	 	(\ref{equation:TSP:ILP3}) are the {\em subtour elimination constraints},
	 	which forbid solutions consisting of several disconnected tours,
	and finally (\ref{equation:TSP:ILP4}) defines the {\em integrality constraints}. 
Note also that some subtour elimination constraints are redundant:
For the vertex sets $S \subset V$, $S \neq\emptyset$, and $S^\prime = V \backslash S$ 
we get pairs of subtour elimination constraints 
both enforcing the connection of $S$ and $S^\prime$.

The established standard approach to solve TSP to optimality, as pursued successfully
during the last 30+ years,
is a branch-and-cut approach, which solves the LP-relaxation obtained by
relaxing the integrality constraints (\ref{equation:TSP:ILP4})
into $x_e \in [0,1]$.
In each iteration of the underlying branch-and-bound scheme
cutting planes are generated, i.e.\ constraints that are violated by the current
fractional solution, but not necessarily by any feasible integer solution.
Since there exists an exponential number of subsets $S \subset V$ implying subtour elimination constraints	 (\ref{equation:TSP:ILP3}),
the computation starts with a small collection of subsets $S \subset V$ (or none at all),
and identifies violated subtour elimination constraints as cutting planes
in the so-called separation problem.
Moreover, a wide range of other cutting plane families were developed in the literature
together with heuristic and exact algorithms to find them
(see e.g.\ \cite[ch.~58]{sch03}, \cite{TheTravelingSalesmanProblemAComputionalStudy}).
Also the undisputed champion among all TSP codes,
the famous {\em Concorde} package (see~\cite{TheTravelingSalesmanProblemAComputionalStudy}),
is based on this principle.

\medskip
In this paper we introduce and examine another concept for solving the TSP.
In Section~\ref{section:newApproachToSolveTheTSPToOptimality} we introduce the basic idea of our approach. 
Some improvement strategies follow in Section~\ref{sec:gensub}
with our best approach presented in Subsection~\ref{sec:hier}.
Since the main contribution of this paper are computational experiments, we discuss them in detail in Section~\ref{subsection:computationalResults}. The common details of all these tests will be given in Subsection~\ref{sec:compsetup}. In Section~\ref{section:someTheoreticalResultsAndFurtherEmpiricalObservations}, we present some theoretical results and further empirical observations.
Finally, we provide an Appendix with illustrations, graphs and two summarizing tables (Tables~\ref{table:resultsForTheMainApproachAndForDifferentVariantsOfTheApproachWhichUsesTheHierarchicalClustering} and \ref{table:comparisonBetweenDifferentVariantsOfOurApproach}).

\section{General solution approach}
	\label{section:newApproachToSolveTheTSPToOptimality}

Clearly, the performance of the above branch-and-cut approach
depends crucially on the performance of the used LP-solver.
Highly efficient LP-solvers have been available for quite some time,
but also ILP-solvers have improved rapidly during the last decades and reached an impressive performance.
This motivated the idea of a very simple approach for solving TSP without
using LP-relaxations explicitly.

The general approach works as follows
(see Algorithm~\ref{algorithm:mainIdeaOfOurApprach}).
First, we relax all subtour elimination constraints (\ref{equation:TSP:ILP3}) from the model
and solve the remaining ILP model (corresponding to a {\em weighted 2-matching problem}).
Then we check if the obtained integer solution contains subtours.
If not, the solution is an optimal TSP tour.
Otherwise, we find all subtours in the integral solution
(which can be done by a simple scan)
and add the corresponding subtour elimination constraints to the model, 
each of them represented by the subset of vertices in the corresponding subtour.
The resulting enlarged ILP model is solved again to optimality.
Iterating this process clearly leads to an optimal TSP tour.

\begin{algorithm}
\begin{algorithmic}[1]
				\REQUIRE TSP instance
				\ENSURE an optimal TSP tour
				\STATE define current model as
    (\ref{equation:TSP:ILP1}), (\ref{equation:TSP:ILP2}), (\ref{equation:TSP:ILP4});
    			\label{algorithm:mainIdeaOfOurApprach:startILP}
				\REPEAT
					\STATE solve the current model to optimality by an ILP-solver;					 \label{algorithm:mainIdeaOfOurApprach:solveTheCurrentProblemToOptimalityByUsingAnILPSolver}
					\IF{solution contains no subtour}
                        \STATE set the solution as optimal tour;
                     \ELSE
						\STATE find all subtours of the solution and add the corresponding subtour elimination constraints into the model;
						 \label{algorithm:mainIdeaOfOurApprach:findAllSubtoursAndIntroduceTheCorrespondingSubtourConstraintsIntoTheModel}
					\ENDIF
				\UNTIL{optimal tour found};
\end{algorithmic}
\caption{Main idea of our approach.}
\label{algorithm:mainIdeaOfOurApprach}
\end{algorithm}

Every execution of the ILP-solver (see line~\ref{algorithm:mainIdeaOfOurApprach:solveTheCurrentProblemToOptimalityByUsingAnILPSolver})
will be called an {\em iteration}.
We define the {\em set of violated subtour elimination constraints} as the set of all included subtour elimination constraints which were violated in an iteration (see line~\ref{algorithm:mainIdeaOfOurApprach:findAllSubtoursAndIntroduceTheCorrespondingSubtourConstraintsIntoTheModel}). Figures~\ref{figure:instance:REA150} and \ref{figure:instanceREA150MainIdeaOfOurApproachIteration0}~--~\ref{figure:instanceREA150MainIdeaOfOurApproachIteration11} in the Appendix illustrate a problem instance and the execution of the algorithm on this instance respectively. 


\medskip
It should be pointed out that the main motivation of this framework
is its simplicity.
The separation of subtour elimination constraints for fractional solutions amounts to the solution of a max-flow or min-cut problem.
Based on the procedure by Padberg and Rinaldi~\cite{pari90},
extensive work has been done
to construct elaborated algorithms for performing this task efficiently.
On the contrary, violated subtour elimination constraints of integer solutions are trivial to find.
Moreover, we refrain from using any other additional inequalities
known for classical branch-and-cut algorithms,
which might also be used to speed up our approach, since we want to underline
the strength of modern ILP-solvers in connection with a
refined subtour selection process (see Section~\ref{sec:hier}).

This approach for solving TSP is clearly not new but was available since the earliest ILP formulation
going back to \cite{SolutionOfALargeScaleTravelingSalesmanProblem} and can be seen as folklore
nowadays.
Several authors followed the concept of generating integer solutions
for some kind of relaxation of an ILP formulation and iteratively adding
violated integer subtour elimination constraints.
However, it seems that the lack of fast ILP-solvers prohibited its direct application
in computational studies although it was used in an artistic context by~\cite{ConnectingTheDotsTheInsAndOutsOfTSPArt}.

Miliotis~\cite{mil76} also concentrated on generating
integer subtour elimination constraints, but within a fractional LP framework.
The classical paper by Crowder and Padberg \cite{SolvingLargeScaleSymmetricTravellingSalesmanProblemsToOptimality}
applies the iterative generation of integer subtour elimination constraints as a second part
of their algorithm after generating fractional cutting planes in the first part
to strengthen the LP-relaxation.
They report that not more than three iterations of the ILP-solver
for the strengthened model
were necessary for test instances up to 318 vertices.
Also Gr\"otschel and Holland~\cite{SolutionOfLargeScaleSymmetricTravellingSalesmanProblems} follow this direction of
first improving the LP-model as much as possible,
e.g.\ by  running preprocessing, fixing certain variables
and strengthening the LP-relaxation by different families
of cutting planes, before
generating integer subtours as last step to find an optimal tour.
It turns out that about half of their test instances
never reach this last phase.
In contrast, we stick to the pure ILP-formulation
without any previous modifications.
		
\medskip
From a theoretical perspective, the generation of subtours involves a certain trade-off.
For an instance $(G, d)$ there exists a minimal set of subtours ${\cal S}^*$,
such that the ILP model with only those subtour elimination constraints implied by ${\cal S}^*$
yields an overall feasible, and thus optimal solution.
However, in practice we can only find collections of subtours larger than ${\cal S}^*$
by adding subtours in every iteration until we reach optimality.
Thus, we can either collect as many subtours as possible in each iteration,
which may decrease the number of iterations but increases the running time of
the ILP-solver because of the larger number of constraints.
Or we try to control the number of subtour elimination constraints added to the model
by trying to judge their relevance and possibly remove some of them later,
which keeps the ILP-model smaller
but may increase the number of iterations.
In the following we describe various strategies to find the ``right'' subtours.

\subsection{Representation of subtour elimination constraints}
\label{sec:repsub}

The subtour elimination constraints~(\ref{equation:TSP:ILP3})
can be expressed equivalently by the following cut constraints:
\begin{equation}\label{eq:subtour2}
\sum_{\substack{{e=(u,v) \in E}\\{u \in S, v \not\in S}}}{x_e} \ \geq\  2 \qquad	
\ \forall\;S \subset V, S \neq\emptyset
\end{equation}
Although mathematically equivalent, the two ways of forbidding a subtour in $S$
may result in quite different performances of the ILP-solver.

It was observed that in general the running time for solving an ILP increases with the number of non-zero entries of the constraint matrix.
Hence, we also tested a hybrid variant which chooses between (\ref{equation:TSP:ILP3})
and (\ref{eq:subtour2}) by picking for each considered set $S$
the version with the smaller number of nonnegative coefficients on the left-hand side
as follows:
\begin{equation}\label{eq:subtour3}
	\begin{array}{llll}
		\sum_{\substack{{e=(u,v) \in E}\\{u,v \in S}}}{x_e} \ \leq \ |S| - 1
			& \qquad \multirow{2}{*}{$\ \forall\;S \subset V, S \neq\emptyset$} \qquad
				& \mbox{if}
					& |S| \leq \frac{2 n + 1}{3}\\
		\sum_{\substack{{e=(u,v) \in E}\\{u \in S, v \not\in S}}}{x_e} \ \geq \ 2
			&
				& \mbox{if}
					& |S| > \frac{2 n + 1}{3}\\
	\end{array}
\end{equation}
We performed computational tests of our approach
to compare the three representations of subtour elimination constraints,
namely (\ref{equation:TSP:ILP3}), (\ref{eq:subtour2}) and (\ref{eq:subtour3}),
and list the results in
Table~\ref{table:comparisonOfTheBehaviourOfTheAlgorithmIfUsingDifferentRepresentationsOfSubtourConstraints}.
Technical details about the setup of the experiments can be found in Subsection~\ref{sec:compsetup}.

\begin{table}[!ht]
	\begin{center}
		\begin{tabular}{l||rrr|rrr|rrr}
			{\bf instance} & \multicolumn{3}{c|}{\bf s.e.c. as in~(\ref{equation:TSP:ILP3})} & \multicolumn{3}{c|}{\bf s.e.c. as in~(\ref{eq:subtour2})} & \multicolumn{3}{c}{\bf s.e.c. as in~(\ref{eq:subtour3})}\\
			& {\bf sec.} & {\bf \#i.} & {\bf \#c.} & {\bf sec.} & {\bf \#i.} & {\bf \#c.} & {\bf sec.} & {\bf \#i.} & {\bf \#c.}\\
			\hline\hline
			kroA150
				& 89
					& 12
						& 82
							& 75
								& 12
									& 82
										& {\bf 62}
											& {\bf 12}
												& {\bf 82}\\
			
			kroB150
				& {\bf 52}
					& {\bf 13}
						& {\bf 77}
							& 237
								& 13
									& 77
										& 54
											& 13
												& 77\\
			
			u159
				& {\bf 9}
					& {\bf 5}
						& {\bf 39}
							& 13
								& 5
									& 39
										& {\bf 9}
											& {\bf 5}
												& {\bf 38}\\
			
			brg180
				& 62
					& 14
						& 56
							& {\bf 36}
								& {\bf 5}
									& {\bf 29}
										& 64
											& 16
												& 67\\
			
			kroA200
				& 2153
					& 11
						& 95
							& {\bf 1833}
								& {\bf 11}
									& {\bf 95}
										& 2440
											& 11
												& 95\\
			
			kroB200
				& 45
					& 7
						& 65
							& 146
								& 7
									& 65
										& {\bf 37}
											& {\bf 7}
												& {\bf 65}\\
			
			tsp225
				& {\bf 149}
					& {\bf 15}
						& {\bf 102}
							& 376
								& 16
									& 105
										& 155
											& 16
												& 106\\
			
			a280
				& {\bf 114}
					& {\bf 10}
						& {\bf 59}
							& 249
								& 10
									& 56
										& 132
											& 10
												& 63\\
			
			lin318
				& 7171
					& 13
						& 177
							& 8201
								& 13
									& 177
										& {\bf 7158}
											& {\bf 13}
												& {\bf 177}\\
			
			gr431
				& 5973
					& 22
						& 186
							& 19111
								& 22
									& 187
										& {\bf 5925}
											& {\bf 22}
												& {\bf 186}\\
			
			pcb442
				& 4406
					& 43
						& 215
							& 6186
								& 41
									& 197
										& {\bf 2393}
											& {\bf 43}
												& {\bf 207}\\
			
			gr666
				& {\bf 33259}
					& {\bf 14}
						& {\bf 216}
							& 189421
								& 14
									& 217
										& 40111
											& 14
												& 216\\
			
			\hline
			
			{\em mean ratio (sec.)}
				& \multicolumn{3}{c}{}
					& \multicolumn{3}{c}{\em 2.305960}
						& \multicolumn{3}{c}{\em 0.971694}\\
			
			\hline\hline
			
			RE\_A\_150
				& {\bf 23}
					& {\bf 12}
						& {\bf 61}
							& 65
								& 12
									& 61
										& 26
											& 12
												& 61\\
			
			RE\_A\_200
				& 81
					& 15
						& 84
							& 139
								& 15
									& 84
										& {\bf 76}
											& {\bf 15}
												& {\bf 84}\\
			
			RE\_A\_250
				& 156
					& 14
						& 82
							& 208
								& 14
									& 82
										& {\bf 133}
											& {\bf 14}
												& {\bf 82}\\
			
			RE\_A\_300
				& {\bf 534}
					& {\bf 14}
						& {\bf 123}
							& 4819
								& 14
									& 123
										& 692
											& 14
												& 123\\
			
			RE\_A\_350
				& {\bf 404}
					& {\bf 9}
						& {\bf 110}
							& 789
								& 9
									& 110
										& 650
											& 9
												& 110\\
			
			RE\_A\_400
				& 49234
					& 16
						& 179
							& 247511
								& 16
									& 179
										& {\bf 24619}
											& {\bf 16}
												& {\bf 179}\\
			
			RE\_A\_450
				& 4666
					& 8
						& 117
							& 13806
								& 8
									& 117
										& {\bf 3022}
											& {\bf 8}
												& {\bf 117}\\
			
			RE\_A\_500
				& 68215
					& 12
						& 167
							& 155977
								& 12
									& 167
										& {\bf 30809}
											& {\bf 12}
												& {\bf 167}\\
			
			\hline
			
			{\em mean ratio (sec.)}
				& \multicolumn{3}{c}{}
					& \multicolumn{3}{c}{\em 3.390678}
						& \multicolumn{3}{c}{\em 0.928176}\\
			
			\hline\hline
			
			{\em mean ratio all}
				& \multicolumn{3}{c}{}
					& \multicolumn{3}{c}{\em 2.739847}
						& \multicolumn{3}{c}{\em 0.954287}\\
			
		\end{tabular}
	\end{center}
	\caption{Comparison of the behavior of the algorithm for different representations of subtour elimination constraints. {\em Mean ratios} refer to the arithmetic
 means over ratios between the running times of the approaches using the subtour elimination constraints represented as in~(\ref{eq:subtour2}) and (\ref{eq:subtour3}) respectively and the running time of the approach using the subtour elimination constraints represented as in~(\ref{equation:TSP:ILP3}). ``sec.'' is the time in seconds, ``\#i.'' the number of iterations and ``\#c.'' the number of subtour elimination constraints added to the ILP before starting the last iteration.}
	 \label{table:comparisonOfTheBehaviourOfTheAlgorithmIfUsingDifferentRepresentationsOfSubtourConstraints}
\end{table}

It turned out that the three versions sometimes (but not always) lead to huge differences in running time (up to a factor of 5).
This is an interesting experience that should be taken into consideration also in other computational studies.
From our limited experiments it could be seen that version (\ref{eq:subtour2})
was inferior most of the times (with sometimes huge deviations) whereas only a small dominance of the hybrid variant (\ref{eq:subtour3}) in comparison with the standard version (\ref{equation:TSP:ILP3}) could be observed.
This is due to the small size of  most subtours occurring during the solution process
(the representation (\ref{equation:TSP:ILP3}) equals to the representation (\ref{eq:subtour3}) in these cases). But since also bigger subtours can occur (mostly in the last iterations), we use the representation (\ref{eq:subtour3})
for all further computational tests.
For more details about different ILP-models 
see~\cite{AComparativeAnalysisOfSeveralAsymmetricTravelingSalesmanProblemFormulations}.

\section{Generation of subtours}
\label{sec:gensub}

As pointed out above, the focus of our attention lies in the
generation and selection of a ``good'' set of subtour elimination constraints,
including as many as possible of those required by the ILP-solver
to determine an optimal solution which is also feasible for TSP,
but as few as possible of all others which only slow down the
performance of the ILP-solver.

Trying to strike a balance between these two goals we followed
several directions, some of them motivated by theoretical results,
others by visually studying plots of all subtours generated during the
execution of Algorithm \ref{algorithm:mainIdeaOfOurApprach}.

\subsection{Subtour elimination constraints from suboptimal integer solutions}
\label{sec:suboptimal}

Many ILP-solvers report all feasible integer solutions found during the underlying branch-and-bound process.
In this case, we can also add all corresponding subtour elimination constraints to the model.
These constraints can be considered simply as part of the set of violated subtour elimination constraints.
Not surprisingly, these additional constraints always lead to a
decrease in the number of iterations for the overall computation
and to an increase in the total number of subtour elimination constraints generated before reaching
optimality (see Table~\ref{table:comparisonOfTheBehaviourOfTheAlgorithmInTheCaseOfUsingAllConstraintsGeneratedFromAllFeasibleSolutionsFoundDuringTheSolvingProcessAndInTheCaseOfUsingOnlyTheConstraintsGeneratedFromTheFinalILP}).
While the time consumed in each iteration is likely to increase,
it can also be observed that the overall running time
is often decreased significantly by adding all detected subtours to the model.
On the other hand, for the smaller number of instances where this is not the case, only relatively
modest increases of running times are incurred.
Therefore, we stick to adding all detected subtour elimination constraints
for the remainder of the paper.
The algorithm in this form will be called {\em \mainApproach}.

\begin{table}[!ht]
	\begin{center}
		\begin{tabular}{l||rrr|rrr}
			& \multicolumn{3}{c|}{\bf only subtours} & \multicolumn{3}{c}{\bf all subtours:}\\&\multicolumn{3}{c|}{\bf from ILP-optima}&  \multicolumn{3}{c}{\bf \mainApproach} \\
			{\bf instance} & {\bf sec.} & {\bf \#i.} & {\bf \#c.} & {\bf sec.} & {\bf \#i.} & {\bf \#c.}\\
			\hline\hline
			kroA150
				& 62
					& 12
						& 82
							& {\bf 19}
								& {\bf 7}
									& {\bf 136}\\
			
			kroB150
				& {\bf 54}
					& {\bf 13}
						& {\bf 77}
							& 179
								& 8
									& 148\\
			
			u159
				& 9
					& 5
						& 38
							& {\bf 6}
								& {\bf 4}
									& {\bf 49}\\
			
			brg180
				& 64
					& 16
						& 67
							& {\bf 44}
								& {\bf 4}
									& {\bf 103}\\
			
			kroA200
				& 2440
					& 11
						& 95
							& {\bf 677}
								& {\bf 8}
									& {\bf 237}\\
			
			kroB200
				& 37
					& 7
						& 65
							& {\bf 31}
								& {\bf 5}
									& {\bf 121}\\
			
			tsp225
				& {\bf 155}
					& {\bf 16}
						& {\bf 106}
							& 178
								& 9
									& 261\\
			
			a280
				& {\bf 132}
					& {\bf 10}
						& {\bf 63}
							& 157
								& 11
									& 143\\
			
			lin318
				& 7158
					& 13
						& 177
							& {\bf 6885}
								& {\bf 8}
									& {\bf 357}\\
			
			gr431
				& 5925
					& 22
						& 186
							& {\bf 2239}
								& {\bf 9}
									& {\bf 453}\\
			
			pcb442
				& {\bf 2393}
					& {\bf 43}
						& {\bf 207}
							& 2737
								& 11
									& 501\\
			
			gr666
				& 40111
					& 14
						& 216
							& {\bf 17711}
								& {\bf 8}
									& {\bf 789}\\
			
			\hline
			
			{\em mean ratio (sec.)}
				& \multicolumn{6}{c}{\em 0.946130}\\
			
			\hline\hline
			
			RE\_A\_150
				& 26
					& 12
						& 61
							& {\bf 23}
								& {\bf 8}
									& {\bf 100}\\
			
			RE\_A\_200
				& 76
					& 15
						& 84
							& {\bf 72}
								& {\bf 7}
									& {\bf 163}\\
			
			RE\_A\_250
				& {\bf 133}
					& {\bf 14}
						& {\bf 82}
							& 138
								& 9
									& 186\\
			
			RE\_A\_300
				& {\bf 692}
					& {\bf 14}
						& {\bf 123}
							& 866
								& 6
									& 295\\
			
			RE\_A\_350
				& 650
					& 9
						& 110
							& {\bf 411}
								& {\bf 5}
									& {\bf 252}\\
			
			RE\_A\_400
				& 24619
					& 16
						& 179
							& {\bf 8456}
								& {\bf 8}
									& {\bf 454}\\
			
			RE\_A\_450
				& 3022
					& 8
						& 117
							& {\bf 2107}
								& {\bf 5}
									& {\bf 279}\\
			
			RE\_A\_500
				& 30809
					& 12
						& 167
							& {\bf 15330}
								& {\bf 6}
									& {\bf 436}\\
			
			\hline
			
			{\em mean ratio (sec.)}
				& \multicolumn{6}{c}{\em 0.786451}\\
			
			\hline\hline
			
			{\em mean ratio all}
				& \multicolumn{6}{c}{\em 0.882259}\\
			
		\end{tabular}
	\end{center}
	\caption{Using all constraints generated from all feasible solutions found during the solving process vs.\ using only the constraints generated from the final ILP solutions of each iteration. 
	{\em Mean ratios} refer to the arithmetic means over ratios between the running times of 
	\mainApproach\ over the other approach. 
	``sec.'' is the time in seconds, ``\#i.'' the number of iterations and ``\#c.'' the number of subtour elimination constraints added to the ILP before starting the last iteration.}
 \label{table:comparisonOfTheBehaviourOfTheAlgorithmInTheCaseOfUsingAllConstraintsGeneratedFromAllFeasibleSolutionsFoundDuringTheSolvingProcessAndInTheCaseOfUsingOnlyTheConstraintsGeneratedFromTheFinalILP}
\end{table}

\ifdefined\FULLVERSION

\subsection{Subtours of size 3}
\label{sec:subtoursOfSize3}
The next idea we tried was to add subtour inequalities corresponding to some subtours of size $3$ into the model before starting the iteration process 
(i.e.\ in line~\ref{algorithm:mainIdeaOfOurApprach:startILP} 
of Algorithm~\ref{algorithm:mainIdeaOfOurApprach}). 
This idea was motivated by the observations that in many examples smaller subtours (with respect to their cardinality) occur more often than the larger ones. However, there are $\binom{|V|}{3}$ such subtours and thus we should concentrate only on a relevant subset of them. After studying our computational tests we decided to use the shortest ones with respect to their length. Table~\ref{table:usingNoSubtoursOfSize3VsUsingTheShortestSubtoursOfSize3ForGenerationOfSubtoursConstraintsBeforeStartingTheSolvingProcess} summarizes our computational results and it can be seen that this idea actually tends to slow down our approach. Thus we did not follow it any more.
\begin{table}[!ht]
	\begin{center}
		\begin{tabular}{l||rrr|rrr|rrr}
			{\bf instance} & \multicolumn{3}{c|}{\bf $p = 0$} & \multicolumn{3}{c|}{\bf $p = \frac{1}{10000}$} & \multicolumn{3}{c}{\bf $p = \frac{1}{1000}$}\\
			& {\bf sec.} & {\bf \#i.} & {\bf \#c.} & {\bf sec.} & {\bf \#i.} & {\bf \#c.} & {\bf sec.} & {\bf \#i.} & {\bf \#c.}\\
			\hline\hline
			kroA150
				& {\bf 19}
					& {\bf 7}
						& {\bf 136}
							& {\bf 19}
								& {\bf 7}
									& {\bf 97}
										& 40
											& 5
												& 116\\
			
			kroB150
				& 179
					& 8
						& 148
							& {\bf 71}
								& {\bf 7}
									& {\bf 178}
										& 134
											& 5
												& 105\\
			
			u159
				& {\bf 6}
					& {\bf 4}
						& {\bf 49}
							& 8
								& 4
									& 46
										& {\bf 6}
											& {\bf 3}
												& {\bf 24}\\
			
			brg180
				& 44
					& 4
						& 103
							& {\bf 34}
								& {\bf 15}
									& {\bf 108}
										& 82
											& 9
												& 270\\
			
			kroA200
				& 677
					& 8
						& 237
							& 879
								& 5
									& 157
										& {\bf 504}
											& {\bf 4}
												& {\bf 133}\\
			
			kroB200
				& {\bf 31}
					& {\bf 5}
						& {\bf 121}
							& 32
								& 5
									& 61
										& 43
											& 5
												& 60\\
			
			tsp225
				& 178
					& 9
						& 261
							& {\bf 149}
								& {\bf 10}
									& {\bf 224}
										& 167
											& 9
												& 202\\
			
			a280
				& 157
					& 11
						& 143
							& {\bf 138}
								& {\bf 9}
									& {\bf 98}
										& 156
											& 6
												& 101\\
			
			lin318
				& 6885
					& 8
						& 357
							& 5360
								& 8
									& 302
										& {\bf 1435}
											& {\bf 8}
												& {\bf 291}\\
			
			gr431
				& {\bf 2239}
					& {\bf 9}
						& {\bf 453}
							& 3196
								& 10
									& 534
										& 3648
											& 10
												& 571\\
			
			pcb442
				& {\bf 2737}
					& {\bf 11}
						& {\bf 501}
							& 3483
								& 15
									& 414
										& 3989
											& 14
												& 466\\
			
			gr666
				& {\bf 17711}
					& {\bf 8}
						& {\bf 789}
							& --
								& --
									& --
										& --
											& --
												& --\\
			
			\hline
			
			{\em mean ratio}
				& \multicolumn{3}{c}{}
					& \multicolumn{3}{c}{\em 1.002535}
						& \multicolumn{3}{c}{\em 1.188732}\\
			
			\hline\hline
			
			RE\_A\_150
				& {\bf 23}
					& {\bf 8}
						& {\bf 100}
							& 30
								& 7
									& 130
										& 30
											& 6
												& 77\\
			
			RE\_A\_200
				& 72
					& 7
						& 163
							& 74
								& 8
									& 135
										& {\bf 57}
											& {\bf 6}
												& {\bf 76}\\
			
			RE\_A\_250
				& {\bf 138}
					& {\bf 9}
						& {\bf 186}
							& 155
								& 7
									& 163
										& 140
											& 6
												& 109\\
			
			RE\_A\_300
				& {\bf 866}
					& {\bf 6}
						& {\bf 295}
							& 884
								& 6
									& 203
										& 1344
											& 7
												& 211\\
			
			RE\_A\_350
				& {\bf 411}
					& {\bf 5}
						& {\bf 252}
							& 642
								& 6
									& 147
										& 879
											& 6
												& 150\\
			
			RE\_A\_400
				& 8456
					& 8
						& 454
							& 6623
								& 7
									& 285
										& {\bf 4876}
											& {\bf 8}
												& {\bf 296}\\
			
			RE\_A\_450
				& 2107
					& 5
						& 279
							& {\bf 1226}
								& {\bf 4}
									& {\bf 220}
										& 5386
											& 5
												& 215\\
			
			RE\_A\_500
				& 15330
					& 6
						& 436
							& 13473
								& 6
									& 366
										& {\bf 6114}
											& {\bf 5}
												& {\bf 237}\\
			
			\hline
			
			{\em mean ratio}
				& \multicolumn{3}{c}{}
					& \multicolumn{3}{c}{\em 1.035264}
						& \multicolumn{3}{c}{\em 1.291607}\\
			
			\hline\hline
			
			{\em mean ratio all}
				& \multicolumn{3}{c}{}
					& \multicolumn{3}{c}{\em 1.016316}
						& \multicolumn{3}{c}{\em 1.232048}\\
			
		\end{tabular}
	\end{center}
	\caption{Using no subtours of size $3$ vs. using the shortest subtours of size $3$ for generation of subtour constraints before starting the solving process. The parameter $p$ defines the proportion of used subtour constraints. {\em Mean ratios} refer to the arithmetic
 means over ratios between the running times of the particular approaches and the running time of the \mainApproach\ (corresponding to $p = 0$). ``sec.'' is the time in seconds, ``\#i.'' the number of iterations and ``\#c.'' the number of subtour elimination constraints added to the ILP before starting the last iteration. The entries ``--'' by TSPLIB instances cannot be computed with 16 GB RAM.}
	\label{table:usingNoSubtoursOfSize3VsUsingTheShortestSubtoursOfSize3ForGenerationOfSubtoursConstraintsBeforeStartingTheSolvingProcess}
\end{table}

\subsection{Subtour selections}
\label{sec:subtourSelections}
As mentioned above, a large number of subtour inequalities which are not really needed only slow down our approach. Thus we also tried not to use all subtour inequalities we are able to generate during one iteration, but to make a proper selection. We again used our computational tests in order to identify two general properties which seem to point to such ``suitable'' subtour inequalities.
\begin{itemize}
	\item Sort all obtained subtours with respect to their cardinality, chose the smallest ones and added the corresponding subtour inequalities into the model.
	\item Sort all obtained subtours with respect to their length and proceed as above.
\end{itemize}
The corresponding results are summarized in Tables~\ref{table:usingAllSubtoursVsUsingOnlyTheSmallestSubtoursWithRespectToTheirCardinalityForGenerationOfSubtourConstraints} and \ref{table:usingAllSubtoursVsUsingOnlyTheSmallestSubtoursWithRespectToTheirLengthForGenerationOfSubtourConstraints} and it is obvious that this idea does not speed up our approach as intended. Thus we dropped it from our considerations.
\begin{table}[!ht]
	\begin{center}
		\begin{tabular}{l||rrr|rrr|rrr}
			{\bf instance} & \multicolumn{3}{c|}{\bf $p = 1$} & \multicolumn{3}{c|}{\bf $p = \frac{2}{3}$} & \multicolumn{3}{c}{\bf $p = \frac{1}{3}$}\\
			& {\bf sec.} & {\bf \#i.} & {\bf \#c.} & {\bf sec.} & {\bf \#i.} & {\bf \#c.} & {\bf sec.} & {\bf \#i.} & {\bf \#c.}\\
			\hline\hline
			kroA150
				& {\bf 19}
					& {\bf 7}
						& {\bf 136}
							& 34
								& 8
									& 109
										& 69
											& 19
												& 115\\
			
			kroB150
				& 179
					& 8
						& 148
							& {\bf 51}
								& {\bf 8}
									& {\bf 135}
										& 477
											& 15
												& 134\\
			
			u159
				& {\bf 6}
					& {\bf 4}
						& {\bf 49}
							& 30
								& 4
									& 52
										& 19
											& 11
												& 56\\
			
			brg180
				& 44
					& 4
						& 103
							& {\bf 27}
								& {\bf 6}
									& {\bf 77}
										& 59
											& 19
												& 80\\
			
			kroA200
				& {\bf 677}
					& {\bf 8}
						& {\bf 237}
							& 714
								& 7
									& 171
										& 2846
											& 14
												& 131\\
			
			kroB200
				& {\bf 31}
					& {\bf 5}
						& {\bf 121}
							& 39
								& 6
									& 98
										& 89
											& 13
												& 77\\
			
			tsp225
				& 178
					& 9
						& 261
							& {\bf 100}
								& {\bf 14}
									& {\bf 183}
										& 173
											& 34
												& 166\\
			
			a280
				& 157
					& 11
						& 143
							& {\bf 141}
								& {\bf 12}
									& {\bf 154}
										& 239
											& 27
												& 127\\
			
			lin318
				& {\bf 6885}
					& {\bf 8}
						& {\bf 357}
							& 7069
								& 12
									& 367
										& 9444
											& 32
												& 392\\
			
			gr431
				& {\bf 2239}
					& {\bf 9}
						& {\bf 453}
							& 3210
								& 20
									& 522
										& 4924
											& 38
												& 413\\
			
			pcb442
				& 2737
					& 11
						& 501
							& {\bf 1867}
								& {\bf 18}
									& {\bf 384}
										& 5129
											& 85
												& 386\\
			
			gr666
				& 17711
					& 8
						& 789
							& {\bf 7643}
								& {\bf 7}
									& {\bf 505}
										& 71594
											& 25
												& 597\\
			
			\hline
			
			{\em mean ratio}
				& \multicolumn{3}{c}{}
					& \multicolumn{3}{c}{\em 1.252892}
						& \multicolumn{3}{c}{\em 2.488345}\\
			
			\hline\hline
			
			RE\_A\_150
				& {\bf 23}
					& {\bf 8}
						& {\bf 100}
							& 28
								& 9
									& 109
										& 52
											& 17
												& 94\\
			
			RE\_A\_200
				& 72
					& 7
						& 163
							& {\bf 69}
								& {\bf 8}
									& {\bf 134}
										& 112
											& 23
												& 98\\
			
			RE\_A\_250
				& 138
					& 9
						& 186
							& {\bf 131}
								& {\bf 10}
									& {\bf 149}
										& 208
											& 20
												& 119\\
			
			RE\_A\_300
				& 866
					& 6
						& 295
							& {\bf 792}
								& {\bf 10}
									& {\bf 259}
										& 1720
											& 29
												& 293\\
			
			RE\_A\_350
				& {\bf 411}
					& {\bf 5}
						& {\bf 252}
							& 715
								& 7
									& 232
										& 849
											& 19
												& 177\\
			
			RE\_A\_400
				& {\bf 8456}
					& {\bf 8}
						& {\bf 454}
							& 129380
								& 8
									& 311
										& 107987
											& 26
												& 299\\
			
			RE\_A\_450
				& 2107
					& 5
						& 279
							& {\bf 1544}
								& {\bf 7}
									& {\bf 236}
										& 7987
											& 11
												& 238\\
			
			RE\_A\_500
				& 15330
					& 6
						& 436
							& 18594
								& 8
									& 324
										& {\bf 13738}
											& {\bf 16}
												& {\bf 308}\\
			
			\hline
			
			{\em mean ratio}
				& \multicolumn{3}{c}{}
					& \multicolumn{3}{c}{\em 2.878162}
						& \multicolumn{3}{c}{\em 3.354102}\\
			
			\hline\hline
			
			{\em mean ratio all}
				& \multicolumn{3}{c}{}
					& \multicolumn{3}{c}{\em 1.903000}
						& \multicolumn{3}{c}{\em 2.834648}\\
			
		\end{tabular}
	\end{center}
	\caption{Using all subtours vs. using only the smallest subtours with respect to their cardinality for generation of subtour constraints. The parameter $p$ defines the proportion of used subtour constraints. {\em Mean ratios} refer to the arithmetic
 means over ratios between the running times of the particular approaches and the running time of the \mainApproach\ (corresponding to $p = 1$). ``sec.'' is the time in seconds, ``\#i.'' the number of iterations and ``\#c.'' the number of subtour elimination constraints added to the ILP before starting the last iteration.}
	\label{table:usingAllSubtoursVsUsingOnlyTheSmallestSubtoursWithRespectToTheirCardinalityForGenerationOfSubtourConstraints}
\end{table}
\begin{table}[!ht]
	\begin{center}
		\begin{tabular}{l||rrr|rrr|rrr}
			{\bf instance} & \multicolumn{3}{c|}{\bf $p = 1$} & \multicolumn{3}{c|}{\bf $p = \frac{2}{3}$} & \multicolumn{3}{c}{\bf $p = \frac{1}{3}$}\\
			& {\bf sec.} & {\bf \#i.} & {\bf \#c.} & {\bf sec.} & {\bf \#i.} & {\bf \#c.} & {\bf sec.} & {\bf \#i.} & {\bf \#c.}\\
			\hline\hline
			kroA150
				& {\bf 19}
					& {\bf 7}
						& {\bf 136}
							& 41
								& 10
									& 131
										& 46
											& 16
												& 90\\
			
			kroB150
				& {\bf 179}
					& {\bf 8}
						& {\bf 148}
							& 495
								& 7
									& 152
										& 250
											& 16
												& 112\\
			
			u159
				& {\bf 6}
					& {\bf 4}
						& {\bf 49}
							& 14
								& 5
									& 60
										& 23
											& 12
												& 55\\
			
			brg180
				& 44
					& 4
						& 103
							& {\bf 24}
								& {\bf 13}
									& {\bf 86}
										& 161
											& 8
												& 78\\
			
			kroA200
				& {\bf 677}
					& {\bf 8}
						& {\bf 237}
							& 862
								& 6
									& 124
										& 1829
											& 13
												& 132\\
			
			kroB200
				& {\bf 31}
					& {\bf 5}
						& {\bf 121}
							& 59
								& 7
									& 121
										& 79
											& 11
												& 89\\
			
			tsp225
				& 178
					& 9
						& 261
							& {\bf 112}
								& {\bf 13}
									& {\bf 197}
										& 197
											& 32
												& 159\\
			
			a280
				& 157
					& 11
						& 143
							& {\bf 94}
								& {\bf 9}
									& {\bf 101}
										& 212
											& 21
												& 96\\
			
			lin318
				& {\bf 6885}
					& {\bf 8}
						& {\bf 357}
							& 7688
								& 13
									& 355
										& 9593
											& 36
												& 390\\
			
			gr431
				& {\bf 2239}
					& {\bf 9}
						& {\bf 453}
							& 6091
								& 15
									& 565
										& 9434
											& 45
												& 530\\
			
			pcb442
				& 2737
					& 11
						& 501
							& {\bf 2365}
								& {\bf 18}
									& {\bf 487}
										& 5913
											& 70
												& 399\\
			
			gr666
				& 17711
					& 8
						& 789
							& {\bf 14713}
								& {\bf 10}
									& {\bf 735}
										& --
											& --
												& --\\
			
			\hline
			
			{\em mean ratio}
				& \multicolumn{3}{c}{}
					& \multicolumn{3}{c}{\em 1.478194}
						& \multicolumn{3}{c}{\em 2.434945}\\
			
			\hline\hline
			
			RE\_A\_150
				& {\bf 23}
					& {\bf 8}
						& {\bf 100}
							& 24
								& 9
									& 115
										& 45
											& 22
												& 81\\
			
			RE\_A\_200
				& 72
					& 7
						& 163
							& {\bf 60}
								& {\bf 10}
									& {\bf 123}
										& 113
											& 25
												& 108\\
			
			RE\_A\_250
				& {\bf 138}
					& {\bf 9}
						& {\bf 186}
							& {\bf 138}
								& {\bf 7}
									& {\bf 117}
										& 209
											& 22
												& 103\\
			
			RE\_A\_300
				& {\bf 866}
					& {\bf 6}
						& {\bf 295}
							& 1099
								& 10
									& 321
										& 953
											& 23
												& 201\\
			
			RE\_A\_350
				& {\bf 411}
					& {\bf 5}
						& {\bf 252}
							& 876
								& 7
									& 231
										& 934
											& 16
												& 167\\
			
			RE\_A\_400
				& {\bf 8456}
					& {\bf 8}
						& {\bf 454}
							& 29625
								& 9
									& 311
										& 301125
											& 27
												& 378\\
			
			RE\_A\_450
				& {\bf 2107}
					& {\bf 5}
						& {\bf 279}
							& 2926
								& 7
									& 259
										& 4789
											& 14
												& 237\\
			
			RE\_A\_500
				& {\bf 15330}
					& {\bf 6}
						& {\bf 436}
							& 15786
								& 7
									& 329
										& 37460
											& 16
												& 330\\
			
			\hline
			
			{\em mean ratio}
				& \multicolumn{3}{c}{}
					& \multicolumn{3}{c}{\em 1.524891}
						& \multicolumn{3}{c}{\em 6.092589}\\
			
			\hline\hline
			
			{\em mean ratio all}
				& \multicolumn{3}{c}{}
					& \multicolumn{3}{c}{\em 1.496873}
						& \multicolumn{3}{c}{\em 3.975006}\\
			
		\end{tabular}
	\end{center}
	\caption{Using all subtours vs. using only the smallest subtours with respect to their length for generation of subtour constraints. The parameter $p$ defines the proportion of used subtour constraints. {\em Mean ratios} refer to the arithmetic
 means over ratios between the running times of the particular approaches and the running time of the \mainApproach\ (corresponding to $p = 1$). ``sec.'' is the time in seconds, ``\#i.'' the number of iterations and ``\#c.'' the number of subtour elimination constraints added to the ILP before starting the last iteration. The entries ``--'' by TSPLIB instances cannot be computed with 16 GB RAM.}
	\label{table:usingAllSubtoursVsUsingOnlyTheSmallestSubtoursWithRespectToTheirLengthForGenerationOfSubtourConstraints}
\end{table}

\fi

\subsection{Clustering into subproblems}
\label{sec:cluster}

It can be observed that many subtours have a local context,
meaning that a small subset of vertices separated from the remaining vertices
by a reasonably large distance will always be connected by one or more subtours,
independently from the size of the remaining graph (see also Figures~\ref{figure:instance:REA150} and \ref{figure:instanceREA150MainIdeaOfOurApproachIteration0} to~\ref{figure:instanceREA150MainIdeaOfOurApproachIteration11} in the Appendix).
Thus, we aim to identify {\em clusters} of vertices and run the \mainApproach\ on the induced subgraphs with the aim of generating within a very small
running time the same subtours occurring in the execution of the approach on the full graph.
Furthermore, we can use the optimal tour from every cluster to generate a corresponding subtour elimination constraint for the original instance and thus enforce a connection to the remainder of the graph.

For our purposes the clustering algorithm should fulfill the following properties:
\begin{itemize}
	\item {\em clustering quality:} The obtained clusters should correspond well to
	the distance structure of the given graph, as in a classical geographic clustering.
				\item {\em running time:} Should be low relative to the running time required for the main part of the algorithm.				
				\item {\em cluster size:} If clusters are too large, solving the TSP takes too much time.
				If clusters are too small, only few subtour elimination constraints are generated.
	\end{itemize}
			
Clearly, there is a huge body of literature on clustering algorithms (see e.g.~\cite{AlgorithmsForClusteringData}) and selecting one for a given
application will never satisfy all our objectives.
Our main restriction was the requirement of using a clustering algorithm which works
also if the vertices are not embeddable in Euclidean space, i.e.\ only arbitrary edge distances are given.
Simplicity being another goal, we settled for the following approach
described in Algorithm~\ref{algorithm:clusteringAlgorithm}:

\begin{algorithm}
	\begin{algorithmic}[1]
		\REQUIRE Complete graph $G = (V, E)$, where $|V| = n$ and $|E| = m = \frac{n (n - 1)}{2}$, distance function $d \colon E \to R^+_0$ and parameter $c$, where $1 \leq c \leq n$.
		\ENSURE Clustering ${\cal C} = \{V_1, \ldots, V_c\}$, where $V_1 \cup \ldots \cup V_c = V$.
		\STATE sort the edges such that $d_{e_1} \leq \ldots \leq d_{e_m}$;\label{algorithm:clusteringAlgorithm:startSetting1}
		\STATE define $G^\prime = (V^\prime, E^\prime)$ such that $V^\prime = V$ and $E^\prime = \emptyset$;\label{algorithm:clusteringAlgorithm:startSetting2}
		\STATE let $i \defeq 1$;\label{algorithm:clusteringAlgorithm:startSetting3}
		\STATE define ${\cal C} \defeq \big\{\{v_1\}, \ldots, \{v_n\}\big\}$;\label{algorithm:clusteringAlgorithm:nClusters}
		\WHILE{$|{\cal C}| > c$}\label{algorithm:clusteringAlgorithm:while}
			\STATE set $E^\prime \defeq E^\prime \cup \{e_i\}$;\label{algorithm:clusteringAlgorithm:newEdge}
			\STATE set ${\cal C} \defeq \{V_1, \ldots, V_{|\cal C|}\}$, where $V_1, \ldots, V_{|\cal C|}$ are the connected components of graph $G^\prime$;\label{algorithm:clusteringAlgorithm:newClustering}
		\ENDWHILE
	\end{algorithmic}
	\caption{Clustering algorithm.}
	\label{algorithm:clusteringAlgorithm}
\end{algorithm}

First, we fix the number of clusters $c$ with $1 \leq c \leq n$ and sort the edges in increasing order of distances (see line~\ref{algorithm:clusteringAlgorithm:startSetting1}).
Then we start with the empty graph $G^\prime = (V^\prime, E^\prime)$
(line~\ref{algorithm:clusteringAlgorithm:startSetting2}) containing only isolated vertices
(i.e.\ $n$ clusters) and add iteratively edges in increasing order of distances until the desired number of clusters $c$ is reached (see lines~\ref{algorithm:clusteringAlgorithm:while} and \ref{algorithm:clusteringAlgorithm:newEdge}).
In each iteration the current clustering is implied by the connected components of the current graph (see line~\ref{algorithm:clusteringAlgorithm:newClustering}). 
We denote this {\em clustering approach} by $C \mid c$.
Note that this clustering algorithm does not make any assumptions about the underlying TSP instance and does not exploit any structural properties of the {\em Metric TSP} or the
{\em Euclidean TSP}.
			
It was observed in our computational experiments 
that the performance of the TSP algorithm is not very sensitive to
small changes of the cluster number $c$
and thus a rough estimation of $c$ is sufficient.
The behavior of the running time as a function of $c$ can be found for
particular test instances in
Figure~\ref{figure:progressOfTheAmountOfTheComputationTimeTInSecondsOnTheParameterC},
see Section~\ref{sec:compexamp} for further discussion.

\subsection{Restricted clustering}
\label{sec:cluster2}

Although the clustering algorithm (see Algorithm~\ref{algorithm:clusteringAlgorithm}) decreases the computational time of the whole solution process for some test instances,
we observed a certain shortcoming.
There may easily occur clusters consisting of isolated points
or containing only two vertices.
Clearly, these clusters do not contribute any subtour on their own.
Moreover, the degree constraints (\ref{equation:TSP:ILP2}) guarantee that
each such vertex is connected to the remainder of the graph in any case.
The connection of these vertices to some ``neighboring'' cluster
enforced in \mainApproach\ implies that the clustering yields
different subtours for these neighbors and not the violated subtour elimination constraints
arising in \mainApproach.

To avoid this situation, we want to impose a minimum cluster size of $3$.
An easy way to do so is as follows:
After reaching the $c$ clusters, continue to add edges in increasing order of distances (as before),
but add an edge only, if it is incident to one of the vertices in a connected component (i.e.\ cluster)
of size one or two. This means basically that we simply merge these small clusters to their nearest neighbor with respect to the actual clustering. Note that this is a step-by-step process and it can happen that two clusters of size $1$ merge first before merging the resulting pair to its nearest neighboring cluster. The resulting {\em restricted clustering approach} will be denoted by $RC_3 \mid c$.

			
Against our expectations, the computational experiments (see Section~\ref{subsection:computationalResults}) show that this approach often impacts the algorithm in the opposite way (see also Figure~\ref{figure:progressOfTheAmountOfTheComputationTimeTInSecondsOnTheParameterC} and Table~\ref{table:comparisonBetweenDifferentVariantsOfOurApproach} in the Appendix)
if compared for the same original cluster size $c$.

Surprisingly, we could observe an interesting behavior if $c \approx n$.
In this case, the main clustering algorithm (see Algorithm~\ref{algorithm:clusteringAlgorithm}) has almost no effect,
but the ``post-phase'' which enforces the minimum cluster size
yields a different clustering on its own.
This variant often beats the previous standard clustering algorithm with $c \ll n$ 
(see Table~\ref{table:comparisonBetweenDifferentVariantsOfOurApproach} in the Appendix). 
Note that we cannot fix the actual number of clusters $c^\prime$ in this case. But our computational results show that $c^\prime \approx \frac{n}{5}$ usually holds if the points are distributed relatively uniformly in the Euclidean plane and if the distances correspond to their relative Euclidean distances 
(see Figure \ref{figure:progressOfTheNumberOfObtainedClustersCPrimeOnTheNumberOfVerticesNForTheRestrictedClusteringWhereCEqualNInRandomEuclideanGraphs} in the Appendix).

\subsection{Hierarchical clustering}
\label{sec:hier}

It was pointed out in Section~\ref{sec:cluster} that the number of clusters $c$ is
chosen as an input parameter.
The computational experiments in Subsection~\ref{sec:compexamp}
give some indication on the behavior of Algorithm~\ref{algorithm:clusteringAlgorithm} for different values of $c$,
but fail to provide a clear guideline for the selection of $c$.
Moreover, from graphical inspection of test instances,
we got the impression that a larger number of relevant subtour elimination constraints
might be obtained by considering more clusters of moderate size.
In the following we present an idea that takes both of these aspects into account.

In our {\em hierarchical clustering} process denoted by $HC$
we do not set a cluster number $c$,
but let the clustering algorithm continue until all vertices are connected
(this corresponds to $c=1$).
The resulting clustering process can be represented by a binary {\em clustering tree}
which is constructed in a bottom-up way.
The leaves of the tree represent isolated vertices, i.e.\ the $n$ trivial clusters
given at the beginning of the clustering algorithm.
Whenever two clusters are merged by the addition of an edge,
the two corresponding tree vertices are connected to a new common parent vertex
in the tree representing the new cluster.
At the end of this process we reach the root of the clustering tree
corresponding to the complete vertex set.
An example of such a clustering tree is shown in Figures~\ref{figure:exampleIllustratingTheHierarchicalClustering:graph} and~\ref{figure:exampleIllustratingTheHierarchicalClustering:clustering}.
			
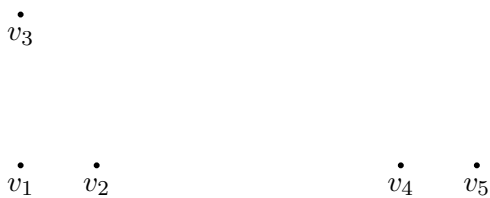
\begin{figure}[!ht]
				\begin{center}
					\begin{tikzpicture}[scale=1]
						\node[circle, draw=black!100, fill=black!100, thick, inner sep=0pt, minimum size=0.5mm, label=below:{\color{black} $v_1$}] (node0) at (0, 0) {};
						\node[circle, draw=black!100, fill=black!100, thick, inner sep=0pt, minimum size=0.5mm, label=below:{\color{black} $v_2$}] (node1) at (1, 0) {};
						\node[circle, draw=black!100, fill=black!100, thick, inner sep=0pt, minimum size=0.5mm, label=below:{\color{black} $v_3$}] (node2) at (0, 2) {};
						\node[circle, draw=black!100, fill=black!100, thick, inner sep=0pt, minimum size=0.5mm, label=below:{\color{black} $v_4$}] (node3) at (5, 0) {};
						\node[circle, draw=black!100, fill=black!100, thick, inner sep=0pt, minimum size=0.5mm, label=below:{\color{black} $v_5$}] (node4) at (6, 0) {};
					\end{tikzpicture}
				\end{center}
				\caption{Example illustrating the hierarchical clustering: Vertices of the TSP instance. Distances between every two vertices correspond to their respective Euclidean distances
in this example.}
				\label{figure:exampleIllustratingTheHierarchicalClustering:graph}
			\end{figure}

			\begin{figure}[!ht]
				\begin{center}
					\begin{tikzpicture}[
							xscale=1.7, yscale=1.7,
							nodeLevel1/.style={circle, draw=black!100, fill=white!0, thick, inner sep=0pt, minimum size=7mm},
							nodeLevel2/.style={circle, draw=black!100, fill=white!0, thick, inner sep=0pt, minimum size=10.5mm},
							nodeLevel3/.style={circle, draw=black!100, fill=white!0, thick, inner sep=0pt, minimum size=13mm},
							nodeLevel4/.style={circle, draw=black!100, fill=white!0, thick, inner sep=0pt, minimum size=20mm}
						]
						
						\node[nodeLevel1] (node1) at (0.00, 0.00) {$\{v_1\}$};
						\node[nodeLevel1] (node2) at (1.00, 0.00) {$\{v_2\}$};
						\node[nodeLevel1] (node3) at (2.00, 0.00) {$\{v_3\}$};
						\node[nodeLevel1] (node4) at (3.00, 0.00) {$\{v_4\}$};
						\node[nodeLevel1] (node5) at (4.00, 0.00) {$\{v_5\}$};
						
						\node[nodeLevel2] (node12) at (0.50, 1.00) {$\{v_1, v_2\}$};
						\node[nodeLevel2] (node45) at (3.50, 1.00) {$\{v_4, v_5\}$};
						\node[nodeLevel3] (node123) at (1.25, 2.00) {$\{v_1, v_2, v_3\}$};
						\node[nodeLevel4] (node12345) at (2.38, 3.00) {$\{v_1, v_2, v_3, v_4, v_5\}$};
						
						\draw [-] (node12345) to (node123);
						\draw [-] (node12345) to (node45);
						\draw [-] (node123) to (node12);
						\draw [-] (node123) to (node3);
						\draw [-] (node12) to (node1);
						\draw [-] (node12) to (node2);
						\draw [-] (node45) to (node4);
						\draw [-] (node45) to (node5);
					\end{tikzpicture}
				\end{center}
				\caption{Example illustrating the hierarchical clustering: Clustering tree.}
				\label{figure:exampleIllustratingTheHierarchicalClustering:clustering}
			\end{figure}
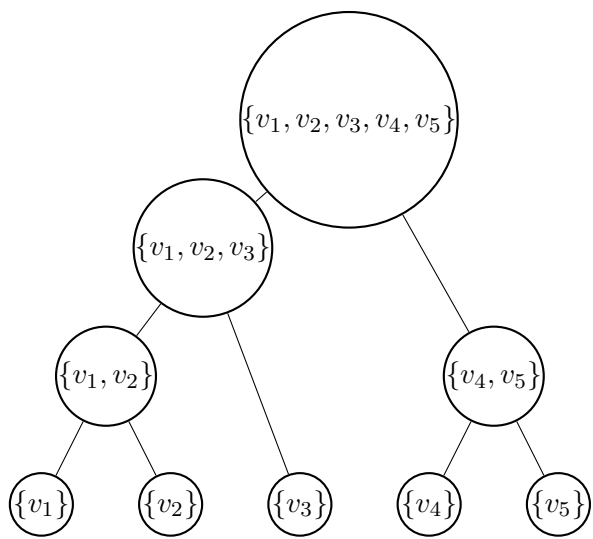

Now, we go through the tree in a bottom-up fashion from the leaves to the root.
In each tree vertex we solve the TSP for the associated cluster,
after both of its child vertices were resolved.
The crucial aspect of our procedure is the following:
All subtour elimination constraints generated during such a TSP solution for a certain cluster
are propagated and added to the ILP model
used for solving the TSP instance of its parent cluster.
Obviously, at the root vertex the full TSP is solved.

The advantage of this strategy is the step-by-step construction
of the violated subtour elimination constraints.
A disadvantage is that many constraints can make sense in the local context
but not in the global one and thus too many constraints could be generated in this way.
Naturally, one pays for the additional subtour elimination constraints
by an increase in computation time required to solve a large number of -- mostly small --
TSP instances.
To avoid the solution of TSPs of the same order of magnitude as the original instance,
it makes sense to impose an upper bound $u$ on the maximum cluster size.
This means that the clustering tree is partitioned into several subtrees by removing
all tree vertices corresponding to clusters of size greater than $u$.
After resolving all these subtrees we collect all generated subtour elimination constraints
and add them to the ILP model for the originally given TSP. This approach will be denoted as $HC \mid u$.
Computational experiments with various choices of $u$
indicated that $u = 4 \frac{n}{\log_2{n}}$ would be a good upper bound.
 	
Let us take a closer look at the problem of including too many subtour elimination constraints which are redundant in the global graph context.
Of course the theoretical ``best'' way would be to check which of the propagated subtour elimination constraints were not used during the runs of the ILP solver and drop them.
To do this, it would be necessary to get this information from the ILP solver
which often is not possible.

However, we can try to approximately identify subtours
which are not only locally relevant in the following way:
All subtour elimination constraints generated in a certain tree vertex, i.e.\ for a certain cluster,
are marked as {\em considered subtour elimination constraints}.
Then we solve the TSP for the cluster of its parent vertex in the tree
without using the subtours marked as {\em considered}.
If we generate such a considered subtour again during the solution
of the parent vertex, we take this as an indicator of global significance
and add the constraint permanently for all following supersets of this cluster.
If we set the upper bound $u$, we take also all subtour elimination constraints found in the biggest solved clusters.
This approach will be denoted as $HCD \mid u$.

Of course, it is only a heuristic rule and one can easily find examples, where this prediction on a subtour's relevance fails, but our experiments indicate that $HCD \mid 4 n / \log_2{n}$ is the best approach we considered.
A comparison with other hierarchical clustering methods for all test instances
can be found in  Table~\ref{table:resultsForTheMainApproachAndForDifferentVariantsOfTheApproachWhichUsesTheHierarchicalClustering} in the Appendix.
It can be seen that without an upper bound we are often not able to find the solution at all (under time and memory constraints we made on the computational experiments).
In the third and fourth column we can see a comparison between approaches both using the upper bound $u = 4 \frac{n}{\log_2{n}}$ where the former collects all detected subtour elimination constraints and the latter allows to drop those which seem to be relevant only in a local context.
Both these methods beat \mainApproach\ (for the comparison of this approach with other presented algorithms see the computational experiments in Section~\ref{subsection:computationalResults}).

\section{Computational experiments}
		\label{subsection:computationalResults}

In the following the computational experiments and their results will be discussed.\ifdefined\FULLVERSION\else{} Additional illustrative material can be found in an accompanying technical report~\cite{GeneratingSubtourEliminationConstraintsForTheTSPFromPureIntegerSolutions},
which is an extended version of this paper.\fi{}

\subsection{Setup of the computational experiments}
\label{sec:compsetup}

All tests were run on an {\em Intel(R) Core(TM) i5-3470 CPU @ 3.20GHz with 16 GB RAM}
under {\em Linux}\footnote{precise version: {\em Linux 3.8.0-29-generic {\#}42{\textasciitilde}precise1-Ubuntu SMP x86\_64 x86\_64 x86\_64 GNU/Linux}} and all programs were implemented in {\em C++}\footnote{precise compiler version: {\em gcc version 4.6.3}} by using
the {\em SCIP} MIP-solver~\cite{SCIPSolvingConstraintIntegerPrograms}
together with {\em CPLEX} as LP-solver\footnote{precise version: {\em SCIP version 3.0.1 [precision: 8 byte] [memory: block] [mode: optimized] [LP solver: CPLEX 12.4.0.0] [GitHash: 9ee94b7] Copyright (c) 2002-2013 Konrad-Zuse-Zentrum f\"ur Informationstechnik Berlin (ZIB)}}.
It has often been discussed in the literature
(see e.g.~\cite{EfficientSeparationRoutinesForTheSymmetricTravelingSalesmanProblemIISeparatingMultiHandleInequalities})
and in personal communications that
ILP-solvers are relatively unrobust and often show high variations in their running time performance,
even if the same instance is repeatedly run on the same hardware and same software environment.
Our first test runs also exhibited deviations up to a factor of $2$
when identical tests were repeated.
Thus we took special care to guarantee the relative reproducibility of the computational
experiments:
No additional swap memory was made available during the tests,
only one thread was used and no other parallel user processes were allowed.
This leads to a high degree of reproducibility in our experiments.
However, this issue makes a comparison to other simple approaches,
which were tested on other computers under other hardware and software conditions,
extremely difficult.

\medskip
We used two groups of test instances:
The first group is taken from the well-known TSPLIB95~\cite{TSPLIB95},
which contains the established benchmarks for TSP and related problems.
From the collection of instances we chose all those with
(i) at least $150$ and at most $1000$ vertices and
(ii) which could be solved in at most 12 hours
by our \mainApproach.
It turned out that 25 instances of the TSPLIB95 fall into this category
(see Table~\ref{table:comparisonBetweenDifferentVariantsOfOurApproach}),
the largest having $783$ vertices.

We also observed some drawbacks of these instances:
Most of them (23 of 25) are defined as point sets in the Euclidean plane with distances
corresponding to the Euclidean metric or as a set of geographical cities,
i.e.\ points on a sphere.
Moreover, they often contain substructures like meshes or sets of colinear points and finally,
since all distances are rounded to the nearest integer, there are many instances which have
multiple optimal solutions.
These instances are relatively unstable with respect to solution time, number of iterations,
and -- important for our approach -- cardinality of the set of violated subtour elimination constraints.
For our approach instances with a mesh geometry (e.g.\ {\em ts225} from \mbox{TSPLIB95})
were especially prone to unstable behavior, such as widely varying running times
for minor changes in the parameter setting.
This seems to be due to the fact that these instances contain many 2-matchings
with the same objective function value\ifdefined\FULLVERSION{} as illustrated in the following example: Consider a $3 \times (2 n + 2)$ mesh graph (see Figure~\ref{figure:ExampleIllustratingTheBehaviorOfOurApproachesByInstancesBasedOnGraphContainingMeshSubstructures}, left graph). It has $2^n$ optimal TSP tours (see~\cite{TosicBodroza:OnTheNumberOfHamiltonianCyclesOfP4TimesPn}). If we fix a subtour on the first $6$ vertices, we obviously have $2^{n - 1}$ optimal TSP tours on the remaining $3 \times \big((2 n + 2) - 2\big)$vertices (see Figure~\ref{figure:ExampleIllustratingTheBehaviorOfOurApproachesByInstancesBasedOnGraphContainingMeshSubstructures}, right graph) and together with the fixed subtour we have $2^{n - 1}$ 2-matchings having the same objective value as an optimal TSP tour on the original graph.
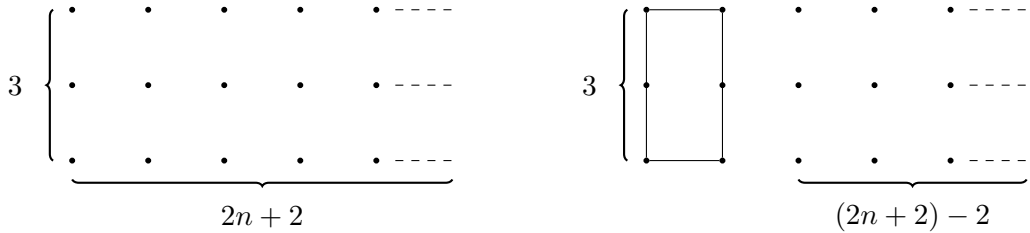
\begin{figure}[htb]
	\begin{center}
		\begin{multicols}{2}
			\begin{tikzpicture}[scale=1]
				\node[circle, draw=black!100, fill=black!100, thick, inner sep=0pt, minimum size=0.5mm] (node0) at (0.00, 0.00) {};
				\node[circle, draw=black!100, fill=black!100, thick, inner sep=0pt, minimum size=0.5mm] (node1) at (1.00, 0.00) {};
				\node[circle, draw=black!100, fill=black!100, thick, inner sep=0pt, minimum size=0.5mm] (node2) at (2.00, 0.00) {};
				\node[circle, draw=black!100, fill=black!100, thick, inner sep=0pt, minimum size=0.5mm] (node4) at (3.00, 0.00) {};
				\node[circle, draw=black!100, fill=black!100, thick, inner sep=0pt, minimum size=0.5mm] (node5) at (4.00, 0.00) {};
				\draw[dashed] (4.25, 0.00) -- (5.00, 0.00);
				\node[circle, draw=black!100, fill=black!100, thick, inner sep=0pt, minimum size=0.5mm] (node6) at (0.00, 1.00) {};
				\node[circle, draw=black!100, fill=black!100, thick, inner sep=0pt, minimum size=0.5mm] (node7) at (1.00, 1.00) {};
				\node[circle, draw=black!100, fill=black!100, thick, inner sep=0pt, minimum size=0.5mm] (node8) at (2.00, 1.00) {};
				\node[circle, draw=black!100, fill=black!100, thick, inner sep=0pt, minimum size=0.5mm] (node9) at (3.00, 1.00) {};
				\node[circle, draw=black!100, fill=black!100, thick, inner sep=0pt, minimum size=0.5mm] (node10) at (4.00, 1.00) {};
				\draw[dashed] (4.25, 1.00) -- (5.00, 1.00);
				\node[circle, draw=black!100, fill=black!100, thick, inner sep=0pt, minimum size=0.5mm] (node11) at (0.00, 2.00) {};
				\node[circle, draw=black!100, fill=black!100, thick, inner sep=0pt, minimum size=0.5mm] (node12) at (1.00, 2.00) {};
				\node[circle, draw=black!100, fill=black!100, thick, inner sep=0pt, minimum size=0.5mm] (node13) at (2.00, 2.00) {};
				\node[circle, draw=black!100, fill=black!100, thick, inner sep=0pt, minimum size=0.5mm] (node14) at (3.00, 2.00) {};
				\node[circle, draw=black!100, fill=black!100, thick, inner sep=0pt, minimum size=0.5mm] (node15) at (4.00, 2.00) {};
				\draw[dashed] (4.25, 2.00) -- (5.00, 2.00);

				\draw [thick, decorate, decoration={brace, mirror}] (0.00, -0.25) -- (5.00, -0.25);
				\node at (2.50, -0.75) {$2 n + 2$};
				\draw [thick, decorate, decoration={brace}] (-0.25, 0.00) -- (-0.25, 2.00);
				\node at (-0.75, 1.00) {$3$};
			\end{tikzpicture}

			\begin{tikzpicture}[scale=1]
				\node[circle, draw=black!100, fill=black!100, thick, inner sep=0pt, minimum size=0.5mm] (node0) at (0.00, 0.00) {};
				\node[circle, draw=black!100, fill=black!100, thick, inner sep=0pt, minimum size=0.5mm] (node1) at (1.00, 0.00) {};
				\node[circle, draw=black!100, fill=black!100, thick, inner sep=0pt, minimum size=0.5mm] (node2) at (2.00, 0.00) {};
				\node[circle, draw=black!100, fill=black!100, thick, inner sep=0pt, minimum size=0.5mm] (node4) at (3.00, 0.00) {};
				\node[circle, draw=black!100, fill=black!100, thick, inner sep=0pt, minimum size=0.5mm] (node5) at (4.00, 0.00) {};
				\draw[dashed] (4.25, 0.00) -- (5.00, 0.00);
				\node[circle, draw=black!100, fill=black!100, thick, inner sep=0pt, minimum size=0.5mm] (node6) at (0.00, 1.00) {};
				\node[circle, draw=black!100, fill=black!100, thick, inner sep=0pt, minimum size=0.5mm] (node7) at (1.00, 1.00) {};
				\node[circle, draw=black!100, fill=black!100, thick, inner sep=0pt, minimum size=0.5mm] (node8) at (2.00, 1.00) {};
				\node[circle, draw=black!100, fill=black!100, thick, inner sep=0pt, minimum size=0.5mm] (node9) at (3.00, 1.00) {};
				\node[circle, draw=black!100, fill=black!100, thick, inner sep=0pt, minimum size=0.5mm] (node10) at (4.00, 1.00) {};
				\draw[dashed] (4.25, 1.00) -- (5.00, 1.00);
				\node[circle, draw=black!100, fill=black!100, thick, inner sep=0pt, minimum size=0.5mm] (node11) at (0.00, 2.00) {};
				\node[circle, draw=black!100, fill=black!100, thick, inner sep=0pt, minimum size=0.5mm] (node12) at (1.00, 2.00) {};
				\node[circle, draw=black!100, fill=black!100, thick, inner sep=0pt, minimum size=0.5mm] (node13) at (2.00, 2.00) {};
				\node[circle, draw=black!100, fill=black!100, thick, inner sep=0pt, minimum size=0.5mm] (node14) at (3.00, 2.00) {};
				\node[circle, draw=black!100, fill=black!100, thick, inner sep=0pt, minimum size=0.5mm] (node15) at (4.00, 2.00) {};
				\draw[dashed] (4.25, 2.00) -- (5.00, 2.00);

				\draw [thick, decorate, decoration={brace, mirror}] (2.00, -0.25) -- (5.00, -0.25);
				\node at (3.50, -0.75) {$(2 n + 2) - 2$};
				\draw [thick, decorate, decoration={brace}] (-0.25, 0.00) -- (-0.25, 2.00);
				\node at (-0.75, 1.00) {$3$};
				
				\draw (node0) -- (node1) -- (node12) -- (node11) -- (node0);
			\end{tikzpicture}
		\end{multicols}
	\end{center}
	\caption{Example illustrating the behavior of our approaches by instances based on graphs containing mesh substructures. Distances between every two vertices correspond to their respective Euclidean distances in this example.}
	\label{figure:ExampleIllustratingTheBehaviorOfOurApproachesByInstancesBasedOnGraphContainingMeshSubstructures}
\end{figure}
Thus\else{}, and thus\fi{} the search process for a feasible TSP tour can vary widely\ifdefined\FULLVERSION\else{} 
(see \cite{GeneratingSubtourEliminationConstraintsForTheTSPFromPureIntegerSolutions}
for more details)\fi{}.

In order to provide further comparisons, we also defined a set of instances based on
{\em random Euclidean graphs}:
In a unit square $[0, 1]^2$ we chose $n$ uniformly distributed points and defined the distance between every two vertices 
as their respective Euclidean distance\footnote{We represented all distances as integers
by scaling with $2^{14}$ and rounding to the nearest integer.}.
These {\em random Euclidean instances} eliminate the potential influence of substructures and always have only one unique optimal solution in all stages of the solving process.
We created 40 such instances named {\em RE\_$X$\_$n$} where $n \in \{150, 200, 250, \dots, 500\}$
indicates the number of vertices and $X \in \{A, B, C, D, E\}$.

\medskip
The running times of our test instances, most of them containing between 150 and 500 vertices,
were often within several hours.
Since we tested many different variants and configurations of our
approach, we selected a subset of these test instances to get faster answers for
determining the best algorithm settings for use in the final tests.
This subset contains 12 (of the 25) TSPLIB instances
and one random instances for every number of vertices $n$ 
(see e.g.\ Table~\ref{table:comparisonOfTheBehaviourOfTheAlgorithmIfUsingDifferentRepresentationsOfSubtourConstraints}.)

\medskip
All our running time tables report the name of the instance,
the running time ({\bf sec.}) in wall-clock seconds (rounded down to nearest integers), the number of iterations ({\bf \#i.}),
i.e.\ the number of calls to the ILP-solver in the main part of our algorithm
(without the TSP solutions for the clusters)
and the number of subtour elimination constraints ({\bf \#c.}) added to the ILP model
in the last iteration, i.e.\ the number of constraints of the model which
yielded an optimal TSP solution.
We often compare two columns of a table by taking the
{\bf mean ratio}, i.e.\
computing the quotient between the running times on the same instance
and taking the arithmetic mean of these quotients.

\subsection{Computational details for selected examples}
\label{sec:compexamp}

Let us now take a closer look at two instances in detail.
While this serves only as an illustration,
we studied lots of these special case scenarios visually
during the development of the clustering approach
to gain a better insight into the structure of subtours generated by \mainApproach.

We selected instances {\em kroB150} and {\em u159} whose vertices are
depicted in Figures~\ref{figure:instance:kroB150} and~\ref{figure:instance:u159} in the Appendix.
Both instances consist of points in the Euclidean plane and the distances between every two vertices  correspond to their respective Euclidean distances, however, they represent two very different instance types: 
The instance {\em kroB150} consists of relatively uniformly distributed points, the instance {\em u159} is more structured and it contains e.g. mesh substructures which are the worst setting for our algorithm (recall Subsection~\ref{sec:compsetup}).

Figure~\ref{figure:progressOfTheAmountOfTheComputationTimeTInSecondsOnTheParameterC}
in the Appendix illustrates the behavior of the running time $t$ in seconds
as a function of the parameter $c$ for the instances {\em kroB150} and {\em u159}.
The full lines correspond to standard clustering approach $C \mid c$ described
in Section~\ref{sec:cluster} (see Algorithm~\ref{algorithm:clusteringAlgorithm}),
while the dashed line corresponds to the restricted clustering $RC_3 \mid c$ 
of Section~\ref{sec:cluster2} with minimum cluster size $3$.
The standard \mainApproach\ without clustering arises for $c=1$.

Instance {\em kroB150} consists of relatively uniformly distributed points in the Euclidean plane, but has a specific property:
By using Algorithm~\ref{algorithm:clusteringAlgorithm} we can observe the occurrence of two
main components also for relatively small coefficient $c$ (already for $c = 6$).
This behavior is rather atypical for random Euclidean graphs, 
cf.~\cite[ch.~13]{RandomGeometricGraphs},
but it provides an advantage for our approach since we do not have to solve cluster instances of the same order of magnitude as the original graph
but have several clusters of moderate size also for small cluster numbers $c$.

Considering the standard clustering approach (Algorithm~\ref{algorithm:clusteringAlgorithm}) in Figure~\ref{figure:progressOfTheAmountOfTheComputationTimeTInSecondsOnTheParameterC}, upper graph,
it can be seen that only a small improvement occurs for $c$ between $2$ and $5$.
Looking at the corresponding clusterings in detail,
it turns out in these cases that there exists only one ``giant connected component''
and all other clusters have size $1$. 
This structure also implies that for the restricted clustering these isolated vertices are merged with the giant component and the effect of clustering is lost completely.
For larger cluster numbers $c$, a considerable speedup is obtained, with some variation, but more or less in the same range for almost all values of $c\geq 6$ (in fact, the giant component splits in these cases). Moreover, the restricted clustering performs roughly as good as the standard clustering for $c \geq 6$.

Instance {\em u159} is much more structured and has many colinear vertices.
Here, we can observe a different behavior.
While the standard clustering is actually beaten by \mainApproach\
for smaller cluster numbers and has a more or less similar performance for
larger cluster numbers,
the restricted clustering is almost consistently better than the other two approaches. For $c$ between $2$ and $10$ there exists a large component containing many mesh substructures which consumes as much computation time as the whole instance.

These two instances give some indication of how to characterize ``good'' instances for our algorithm: They should

\vspace*{-3mm}
\begin{itemize}
\setlength\itemsep{-2mm}
	\item consist of more clearly separated clusters and
	\item not contain mesh substructures and colinear vertices.
\end{itemize}

\subsection{General computational results}
\label{sec:compres}

A summary of the computational results for \mainApproach\ and the most promising variants of clustering based subtour generations can be found in Table~\ref{table:comparisonBetweenDifferentVariantsOfOurApproach}.
For random Euclidean instances we report only the mean values of all five instances
of the same size\ifdefined\FULLVERSION\else{} (detailed results for all random Euclidean instances can be found in our accompanying technical report~\cite{GeneratingSubtourEliminationConstraintsForTheTSPFromPureIntegerSolutions})\fi{}.
It turns out that $HCD \mid 4 \frac{n}{\log_2{n}}$, i.e.\ the hierarchical clustering approach combined with dropping subtour elimination constraints
and fixing them only if they are generated again in the subsequent iteration and with the upper bound on the maximum cluster size $u = 4 \frac{n}{\log_2{n}}$,
gives the best overall performance.
A different behavior can be observed
for instances taken from the TSPLIB and for random Euclidean instances.
On the TSPLIB instances this algorithm $HCD \mid 4 \frac{n}{\log_2{n}}$
is on average about $20\%$ faster than pure \mainApproach\ and beats the other clustering based approaches for most instances.
In those cases, where it is not the best choice, it is usually not far behind.

As already mentioned, best results are obtained with $HCD \mid 4 \frac{n}{\log_2{n}}$ for instances with a strong cluster structure and without mesh substructures (e.g.~{\em pr299}).
For instances with mesh substructures it is difficult to find an optimal 2-matching which is also a TSP tour. 
For random Euclidean instances the results are less clear
but approaches with fixed number of clusters seem to be better then the hierarchical ones.

\medskip
It was a main goal of this study
to find a large number of ``good'' subtour elimination constraints, i.e.\ subtours that are present in the last iteration of the ILP-model of \mainApproach.
Therefore, we show the potentials and limitations of our approach in reaching this goal.
In particular, we will report the relation between the set $S_1$ consisting of all subtours generated by running a hierarchical clustering algorithm with an upper bound $u$ (set as in the computational tests to $u = 4 \frac{n}{\log_2{n}}$) before solving the original problem (i.e.\ the root vertex) and the set $S_2$ containing only the subtour elimination constraints included in the final ILP model of \mainApproach.
We tested the hierarchical clustering with and without the dropping of non-repeated subtours.


		\begin{table}[!ht]
			\begin{center}
				\begin{tabular}{l||cc|cc}
					{\bf instance} & \multicolumn{2}{c|}{\bf \boldmath $HC \mid 4 \frac{n}{\log_2{n}}$ \unboldmath} & \multicolumn{2}{c}{\bf \boldmath $HCD \mid 4 \frac{n}{\log_2{n}}$ \unboldmath}\\
					& {\bf \boldmath $p_{used}$ \unboldmath} & {\bf \boldmath $p_{cov}$ \unboldmath} & {\bf \boldmath $p_{used}$ \unboldmath} & {\bf \boldmath $p_{cov}$ \unboldmath}\\
					\hline\hline
					kroA150
						& 0.262712
							& 0.455882
								& 0.476190
									& 0.367647\\
					
					kroB150
						& 0.222222
							& 0.351351
								& 0.396040
									& 0.270270\\
					
					u159
						& 0.085271
							& 0.448980
								& 0.153226
									& 0.387755\\
					
					brg180
						& 0.133929
							& 0.145631
								& 0.714286
									& 0.145631\\
					
					kroA200
						& 0.209713
							& 0.324895
								& 0.450704
									& 0.270042\\
					
					kroB200
						& 0.206612
							& 0.413223
								& 0.423423
									& 0.388430\\
					
					tsp225
						& 0.134752
							& 0.218391
								& 0.297143
									& 0.199234\\
					
					a280
						& 0.064935
							& 0.314685
								& 0.161943
									& 0.279720\\
					
					lin318
						& 0.234589
							& 0.383754
								& 0.440273
									& 0.361345\\
					
					gr431
						& 0.073701
							& 0.209713
								& 0.221053
									& 0.185430\\
					
					pcb442
						& 0.056759
							& 0.151697
								& 0.133117
									& 0.163673\\
					
					gr666
						& 0.076048
							& 0.271229
								& 0.220379
									& 0.235741\\
					
					\hline
					
					{\em mean}
						& {\em 0.146770}
							& {\em 0.307453}
								& {\em 0.340648}
									& {\em 0.271243}\\
					
					\hline\hline
					
					RE\_A\_150
						& 0.179191
							& 0.310000
								& 0.289157
									& 0.240000\\
					
					RE\_A\_200
						& 0.122642
							& 0.239264
								& 0.212329
									& 0.190184\\
					
					RE\_A\_250
						& 0.120773
							& 0.268817
								& 0.172727
									& 0.204301\\
					
					RE\_A\_300
						& 0.191235
							& 0.325424
								& 0.331915
									& 0.264407\\
					
					RE\_A\_350
						& 0.151274
							& 0.376984
								& 0.285714
									& 0.333333\\
					
					RE\_A\_400
						& 0.170455
							& 0.297357
								& 0.254157
									& 0.235683\\
					
					RE\_A\_450
						& 0.148148
							& 0.415771
								& 0.311178
									& 0.369176\\
					
					RE\_A\_500
						& 0.165485
							& 0.321101
								& 0.276596
									& 0.268349\\
					
					\hline
					
					{\em mean}
						& {\em 0.156150}
							& {\em 0.319340}
								& {\em 0.266722}
									& {\em 0.263179}\\
					
					\hline\hline
	
					{\em mean of all}
						& {\em 0.150522}
							& {\em 0.312207}
								& {\em 0.311078}
									& {\em 0.268018}\\
					
				\end{tabular}
			\end{center}
			\caption{Proportion of used and proportion of covered subtours for our hierarchical clustering approaches with the upper bound $u = 4 \frac{n}{\log_2{n}}$ which (i) does not allow ($HC \mid 4 \frac{n}{\log_2{n}}$) and which (ii) does allow ($HCD \mid 4 \frac{n}{\log_2{n}}$) to drop the unused subtour elimination constraints.}
			 \label{table:proportionOfUsedAndProportionOfCoveredSubtoursForOurHierarchicalClusteringApproaches}
		\end{table}
		

There are two aspects we want to describe:
At first, we want to check whether $S_1$ contains a relevant proportion of ``useful''
subtour contraints, i.e.\ constraints also included in $S_2$,
or whether $S_1$ contains ``mostly useless'' subtours.
Therefore, we report the {\em proportion of used subtours} defined as
\begin{equation}
\label{equation:proportionOfUsedSubtourConstraints}
					p_{used} \defeq \frac{\left|S_1\cap S_2\right|}{\left|S_1\right|}.
				\end{equation}
Secondly, we want to find out to what extend it is possible to find the
``right'' subtours by our approach.
Hence, we define the {\em proportion of covered subtours} defined as
\begin{equation}
\label{equation:proportionOfCoveredSubtourConstraints}
					p_{cov} \defeq \frac{\left|S_1 \cap S_2\right|}{\left|S_2\right|}.
				\end{equation}
The values of $p_{used}$ and $p_{cov}$ are given in
Table~\ref{table:proportionOfUsedAndProportionOfCoveredSubtoursForOurHierarchicalClusteringApproaches}.
It can be seen that empirically there is the chance to find about 26--31\% ($p_{cov}$) of all
required violated subtour elimination constraints. 
If subtour elimination constraints are allowed to be dropped, we are able to find fewer such constraints, but our choice has a better quality ($p_{cov}$ is smaller, but $p_{used}$ is larger),
i.e.\ the solver does not have to work with a large number of constraints which only slow down the solving process and are not necessary to reach an optimal solution.

Furthermore, we can observe a relative big difference between the values of the proportion of used subtour elimination constraints ($p_{used}$) for the TSPLIB instances and for random Euclidean instances if the dropping of redundant constraints is allowed.

\subsection{Adding a starting heuristic}
\label{sec:compstart}

Of course, there are many possibilities of adding improvements to our basic approach.
Lower bounds and heuristics can be introduced, branching rules can be specified, or cutting planes can be generated.
We did not pursue these possibilities since we want to focus on the simplicity of the approach.
Moreover, we wanted to take the ILP solver as a ``black box'' and not interfere with its execution.

Just as an example  which immediately comes to mind,
we added a starting heuristic to give a reasonably good TSP solution as a starting solution
to the ILP solver.
We used the improved version of the classical Lin-Kernighan heuristic
in the code written by Helsgaun~\cite{LKH}.
The computational results reported in  Table~\ref{table:resultsForTheMainApproachUsedWithoutWithTheLinKernighanHeuristicForGeneratingAnInitialSolution}
show that a considerable speedup (roughly a factor of 3, but also much more)
can be obtained in this way.
		
		\begin{table}[!ht]
			\begin{center}
				\begin{tabular}{l||rrr|rrr}
					{\bf instance} & \multicolumn{3}{c|}{\bf without starting} & \multicolumn{3}{c}{\bf with starting} \\ & \multicolumn{3}{c|}{\bf heuristic} & \multicolumn{3}{c}{\bf heuristic}\\
					& {\bf sec.} & {\bf \#i.} & {\bf \#c.} & {\bf sec.} & {\bf \#i.} & {\bf \#c.}\\
					\hline\hline
					kroA150
						& 19
							& 7
								& 136
									& {\bf 16}
										& {\bf 10}
											& {\bf 34}\\
					
					kroB150
						& 179
							& 8
								& 148
									& {\bf 17}
										& {\bf 8}
											& {\bf 104}\\
					
					u159
						& 6
							& 4
								& 49
									& {\bf 4}
										& {\bf 5}
											& {\bf 40}\\
					
					brg180
						& 44
							& 4
								& 103
									& {\bf 0}
										& {\bf 2}
											& {\bf 15}\\
					
					kroA200
						& 677
							& 8
								& 237
									& {\bf 42}
										& {\bf 8}
											& {\bf 135}\\
					
					kroB200
						& 31
							& 5
								& 121
									& {\bf 28}
										& {\bf 6}
											& {\bf 124}\\
					
					tsp225
						& 178
							& 9
								& 261
									& {\bf 73}
										& {\bf 13}
											& {\bf 176}\\
					
					a280
						& 157
							& 11
								& 143
									& {\bf 32}
										& {\bf 8}
											& {\bf 58}\\
					
					lin318
						& 6885
							& 8
								& 357
									& {\bf 4941}
										& {\bf 8}
											& {\bf 259}\\
					
					gr431
						& 2239
							& 9
								& 453
									& {\bf 838}
										& {\bf 10}
											& {\bf 318}\\
					
					pcb442
						& 2737
							& 11
								& 501
									& {\bf 447}
										& {\bf 18}
											& {\bf 207}\\
					
					gr666
						& 17711
							& 8
								& 789
									& {\bf 13225}
										& {\bf 11}
											& {\bf 485}\\
					
					\hline
					
					{\em mean ratio (sec.)}
						& \multicolumn{6}{c}{\em 0.432074}\\
					
					\hline\hline
					
					RE\_A\_150
						& 23
							& 8
								& 100
									& {\bf 14}
										& {\bf 11}
											& {\bf 65}\\
					
					RE\_A\_200
						& 72
							& 7
								& 163
									& {\bf 38}
										& {\bf 11}
											& {\bf 99}\\
					
					RE\_A\_250
						& 138
							& 9
								& 186
									& {\bf 63}
										& {\bf 9}
											& {\bf 124}\\
					
					RE\_A\_300
						& 866
							& 6
								& 295
									& {\bf 146}
										& {\bf 8}
											& {\bf 173}\\
					
					RE\_A\_350
						& 411
							& 5
								& 252
									& {\bf 126}
										& {\bf 6}
											& {\bf 151}\\
					
					RE\_A\_400
						& 8456
							& 8
								& 454
									& {\bf 1274}
										& {\bf 6}
											& {\bf 251}\\
					
					RE\_A\_450
						& 2107
							& 5
								& 279
									& {\bf 482}
										& {\bf 7}
											& {\bf 197}\\
					
					RE\_A\_500
						& 15330
							& 6
								& 436
									& {\bf 1997}
										& {\bf 9}
											& {\bf 241}\\
					
					\hline
					
					{\em mean ratio (sec.)}
						& \multicolumn{6}{c}{\em 0.322231}\\
					
					\hline\hline
					
					{\em mean ratio all}
						& \multicolumn{6}{c}{\em 0.388137}\\
					
				\end{tabular}
			\end{center}
	\caption{Results for \mainApproach\ used without / with the Lin-Kernighan heuristic for generating an initial solution. {\em Mean ratios} refer to the arithmetic
 means over ratios between the running times of the approaches. ``sec.'' is the time in seconds, ``\#i.'' the number of iterations and ``\#c.'' the number of subtour elimination constraints
added to the ILP before starting the last iteration.}
\label{table:resultsForTheMainApproachUsedWithoutWithTheLinKernighanHeuristicForGeneratingAnInitialSolution}
		\end{table}

\section{Some theoretical results and further empirical observations}
\label{section:someTheoreticalResultsAndFurtherEmpiricalObservations}

Although our work mainly aims at computational experiments, we also tried to analyze \mainApproach\ from a theoretical point of view. 
In particular we studied the expected behavior on random Euclidean instances and tried to
characterize the expected cardinality of the minimal set of required subtours ${\cal S^*}$
as defined in Section~\ref{section:newApproachToSolveTheTSPToOptimality}.
It is well known that no polynomially bounded representation of the TSP polytope can be found
and there also exist instances based on a mesh-structure 
for which $\mathbb{E}\left[\left|S^*\right|\right]$
has exponential size, but the question for the expected size of $\left|S^*\right|$ for random Euclidean instances
and thus for the expected number of iterations of our solution algorithm
remains an interesting open problem.

We started with extensive computational tests, some of them presented in 
Figures~\ref{figure:meanNumberOfIterationsMeanLengthOfAnOptimalTSPTourAndMeanLengthOfAnOptimalWeighted2Matching} and \ref{figure:meanNumberOfSubtoursDuringTheMainApproachInRandomEuclideanInstances}
in the Appendix,
to gain empirical evidence on this aspect.
The upper graph in Figure~\ref{figure:meanNumberOfIterationsMeanLengthOfAnOptimalTSPTourAndMeanLengthOfAnOptimalWeighted2Matching} illustrates the mean number of iterations needed by \mainApproach\
to reach optimality for different numbers of vertices $n$
(we evaluated $100$ random Euclidean instances for every value $n$).
The lower graph of Figure~\ref{figure:meanNumberOfIterationsMeanLengthOfAnOptimalTSPTourAndMeanLengthOfAnOptimalWeighted2Matching} shows the mean length of the optimal TSP tour and of the optimal $2$-matching
(i.e.\ the objective value after solving the ILP in the first iteration) by using the same setting.

It was proven back in 1959 that the expected length of an optimal TSP tour is asymptotically
$\beta \sqrt{n}$, where $\beta$ is a constant~\cite{TheShortestPathThroughManyPoints}.
This approach was later generalized for other settings and
other properties of the  square root asymptotic were identified \cite{AMatchingProblemAndSubadditiveEuclideanFunctionals, ProbabilityTheoryOfClassicalEuclideanOptimizationProblems}.
We used these properties to prove the square root asymptotic also for the $2$-matching problem
(cf.\ Figure~\ref{figure:meanNumberOfIterationsMeanLengthOfAnOptimalTSPTourAndMeanLengthOfAnOptimalWeighted2Matching}, lower graph, dashed).

\ifdefined\FULLVERSION

We need some definitions definitions, lemmas and theorems originally introduced by~\cite{AMatchingProblemAndSubadditiveEuclideanFunctionals} and summarized by~\cite{ProbabilityTheoryOfClassicalEuclideanOptimizationProblems} first in order to prove this result.

\begin{definition}[(2-)matching functional and boundary (2-)matching functional]
	\label{definition:2MatchingFunctionalAndBoundary2MatchingFunctional}

		$ $
	
	Let $\mathfrak{F} \defeq \mathfrak{F}(\mbox{\it dim})$ denote the {\em finite subsets of $\rz^{\mbox{\it dim}}$} and let let $\mathfrak{R} \defeq \mathfrak{R}(\mbox{\it dim})$ denote the {\em $\mbox{\it dim}$-dimensional rectangles of $\rz^{\mbox{\it dim}}$}.
	
	Furthermore, let $F \in \mathfrak{F}$ be a point set in $\rz^{\mbox{\it dim}}$ and let $R \in \mathfrak{R}$ be a $\mbox{\it dim}$-dimensional rectangle in $\rz^{\mbox{\it dim}}$ where $\mbox{\it dim} \geq 2$.
	
	And finally, let $d \colon R \times R \to \rz_0^+$ be a metric and let  $G = G(F, R) = \big(V(G), E(G)\big)$ be a complete graph with the vertex set $V(G) = F \cap R$ and with the distances $d(e)$ between every two vertices $u, v \in V(G)$ where $e = (u, v)$.
	
	Then we will denote
	\begin{equation}
		\label{equation:definition:2MatchingFunctionalAndBoundary2MatchingFunctional:MatchingFunctional}
		M(F, R) \defeq M(F \cap R) \defeq \min_{m}{\{OV(G, m) | m \text{ is a matching in } G\}}
	\end{equation}
	the {\em matching functional}.
	
	Furthermore, we will denote
	\begin{equation}
		\label{equation:definition:2MatchingFunctionalAndBoundary2MatchingFunctional:boundaryMatchingFunctional}
		M_B(F, R) \defeq M_B(F \cap R) \defeq \min{\left\{M(F, R), \inf_{\substack{{(F_i)_{i \geq 1} \in \mathcal{F}}\\{\{a_i, b_i\}_{i \geq 1} \in \mathcal{\partial R}}}}{\left\{\sum_i{M(F_i \cup \{a_i, b_i\}, R)}\right\}} \right\}}
	\end{equation}
	the {\em boundary matching functional}. $\mathcal{F}$ stays for the set of all partitions $(F_i)_{i \geq 1}$ of $F$ and $\mathcal{\partial R}$ stays for the set of all sequences of pairs of points $\{a_i, b_i\}_{i \geq 1}$ belonging to the {\em boundary of the rectangle $R$} denoted by $\partial R$. Additionally, we set $d(a, b) = 0$ for all $a, b \in \partial R$.
	
	Similarly, we define the {\em 2-matching functional} $L$ and the {\em boundary 2-matching functional} $L_B$.
	\begin{equation}
		\label{equation:definition:2MatchingFunctionalAndBoundary2MatchingFunctional:2MatchingFunctional}
		L(F, R) \defeq L(F \cap R) \defeq \min_{x}{\{OV(G, x) | x \text{ is a 2-matching in } G)\}}
	\end{equation}
	\begin{equation}
		\label{equation:definition:2MatchingFunctionalAndBoundary2MatchingFunctional:boundary2MatchingFunctional}
		L_B(F, R) \defeq L_B(F \cap R) \defeq \min{\left\{L(F, R), \inf_{\substack{{(F_i)_{i \geq 1} \in \mathcal{F}}\\{\{a_i, b_i\}_{i \geq 1} \in \mathcal{\partial R}}}}{\left\{\sum_i{L(F_i \cup \{a_i, b_i\}, R)}\right\}} \right\}}
	\end{equation}
	
	If it is obvious which rectangle $R$ is considered, we also write $M(F)$ for $M(F, R)$, $M_B(F)$ for $M_B(F, R)$, $L(F)$ for $L(F, R)$ and $L_B(F)$ for $L_B(F, R)$.
	
	Finally, we define $L(F, R) = L_B(F, R) = 0$ if $|F \cap R| < 3$
\end{definition}

\begin{definition}[simple subadditivity, geometric subadditivity and superadditivity]
	\label{definition:simpleSubadditivityGeometricSubadditivityAndSuperAdditivity}
	Let $R$ be a rectangle defined as $[0, t]^{\mbox{\it dim}}$ for some positive constant $t \in \rz^+$ partitioned into two rectangles $R_1$ and $R_1$ ($R_1 \cup R_2 = R$). Furthermore, let $F, G$ be finite sets in $[0, t]^{\mbox{\it dim}}$ and let $P(F, R) \colon \mathfrak{F} \times \mathfrak{R} \to \rz_0^+$ be a function.
	
	The function $P$ is {\em simple subadditive} if the following inequality is satisfied:
	\begin{equation}
		\label{equation:definition:simpleSubadditivityGeometricSubadditivityAndSuperAdditivity:simpleSubadditivity}
		P(F \cup G, R) \leq P(F, R) + P(G, R) + C_1 t
	\end{equation}
	where $C_1 \defeq C_1(\mbox{\it dim})$ is a finite constant.
	
	If
	\begin{equation}
		\label{equation:definition:simpleSubadditivityGeometricSubadditivityAndSuperAdditivity:geometricSubadditivity}
		P(F, R) \leq P(F, R_1) + P(F, R_2) + C_2 \mbox{\it diam}(R)
	\end{equation}
	is fulfilled, we will call the function $P$ {\em geometric subadditive}. $\mbox{\it diam}(R)$ denotes the diameter of the rectangle $R$ and $C_2 \defeq C_2(\mbox{\it dim})$ is a finite constant.
	
	If
	\begin{equation}
		\label{equation:definition:simpleSubadditivityGeometricSubadditivityAndSuperAdditivity:superadditivity}
		P(F, R) \geq P(F, R_1) + P(F, R_2)
	\end{equation}
	is satisfied, we will call the function $P$ {\em superadditive}.
\end{definition}

\begin{definition}[subadditive and superadditive Euclidean functional]
	\label{definition:subadditiveAndSuperadditiveEuclideanFunctional}
	Let
	
	\noindent $P(F, R) \colon \mathfrak{F} \times \mathfrak{R} \to \rz_0^+$ be a function satisfying
	\begin{equation}
		\label{equation:definition:subadditiveAndSuperadditiveEuclideanFunctional:condition1}
		\forall R \in \mathfrak{R} \text{, } P(\emptyset, R) = 0 \ \mbox{,}
	\end{equation}
	\begin{equation}
		\label{equation:definition:subadditiveAndSuperadditiveEuclideanFunctional:condition2}
		\forall y \in \rz^{\mbox{\it dim}} \text{, } R \in \mathfrak{R} \text{, } F \subset R\colon P(F, R) = P(F + y, R + y) \ \mbox{,}
	\end{equation}
	\begin{equation}
		\label{equation:definition:subadditiveAndSuperadditiveEuclideanFunctional:condition3}
		\forall \alpha > 0 \text{, } R \in \mathfrak{R} \text{, } F \subset R\colon P(\alpha F, \alpha R) = \alpha P(F, R) \ \mbox{.}
	\end{equation}
	Then we will say that $P$ is {\em translation invariant} (condition~(\ref{equation:definition:subadditiveAndSuperadditiveEuclideanFunctional:condition2})) and {\em homogeneous} (condition~(\ref{equation:definition:subadditiveAndSuperadditiveEuclideanFunctional:condition3})).
	
	We will call $P$ an {\em Euclidean functional}.
	
	Let $R$ be a rectangle defined as $[0, t]^{\mbox{\it dim}}$ for some positive constant $t \in \rz^+$ partitioned into two rectangles $R_1$ and $R_1$ ($R_1 \cup R_2 = R$). If $P$ satisfies the geometric subadditivity (\ref{equation:definition:simpleSubadditivityGeometricSubadditivityAndSuperAdditivity:geometricSubadditivity}), we will say that $P$ is a {\em subadditive Euclidean functional}. If $P$ is supperadditive (\ref{equation:definition:simpleSubadditivityGeometricSubadditivityAndSuperAdditivity:superadditivity}), we will say that $P$ is a {\em superadditive Euclidean functional}.
\end{definition}

	\begin{lemma}[\cite{ProbabilityTheoryOfClassicalEuclideanOptimizationProblems}, originally \cite{AMatchingProblemAndSubadditiveEuclideanFunctionals}]
	\label{lemma:theMatchingFunctionalMAndTheBoundaryMatchingFunctionalMBAreSubadditiveEuclideanFunctionals}
	The matching functional $M$ and the boundary matching functional $M_B$ are subadditive Euclidean functionals.
	\begin{proof}
		See~\cite{ProbabilityTheoryOfClassicalEuclideanOptimizationProblems}.
	\end{proof}
\end{lemma}

	\begin{lemma}[growth bounds -- \cite{ProbabilityTheoryOfClassicalEuclideanOptimizationProblems}, originally \cite{AMatchingProblemAndSubadditiveEuclideanFunctionals}]
	\label{lemma:growthBounds}
	Let $P$ be a subadditive Euclidean functional. Then there exists a finite constant $C_4 \defeq C_4(\mbox{\it dim})$ such that for all $\mbox{\it dim}$-dimensional rectangles of $\rz^{\mbox{\it dim}}$ and for all $F \subset R$ we have
	\begin{equation}
		\label{equation:lemma:growthBounds}
		P(F, R) \leq C_4 \mbox{\it diam}(R) |F|^{\frac{\mbox{\it dim} - 1}{\mbox{\it dim}}} \ \text{.}
	\end{equation}
	\begin{proof}
		See~\cite{ProbabilityTheoryOfClassicalEuclideanOptimizationProblems}.
	\end{proof}
\end{lemma}

	\begin{definition}[smoothness]
	\label{definition:smoothness}
	An Euclidean functional $P$ is {\em smooth} if there is a finite constant $C_3 \defeq C_3(\mbox{\it dim})$ such that for all sets $F, G \in [0, 1]^{\mbox{\it dim}}$ we have
	\begin{equation}
		\label{equation:definition:smoothness}
		\big|P(F \cup G) - P(F)\big| \leq C_3 |G|^\frac{d - 1}{d} \ \text{.}
	\end{equation}
\end{definition}

\begin{definition}[pointwise closeness]
	\label{definition:pointwiseCloseness}
	Say that Euclidean functionals $P$ and $P_B$ are {\em pointwise close} if for all subsets $F \subset [0, 1]^{\mbox{\it dim}}$ we have
	\begin{equation}
		\label{equation:definition:pointwiseCloseness}
		\left|P\left(F, [0, 1]^{\mbox{\it dim}}\right) - P_B\left(F, [0, 1]^{\mbox{\it dim}}\right)\right| = o\left(|F|^\frac{d - 1}{d}\right) \ \text{.}
	\end{equation}
\end{definition}

\fi

\begin{definition}[complete convergence\ifdefined\FULLVERSION\else{}, \cite{ProbabilityTheoryOfClassicalEuclideanOptimizationProblems}\fi]
	\label{definition:completeConvergence}
	A sequence of random variables $X_n$, $n \geq 1$, {\em converges completely (c.c.)} to a constant $C$ if and only if for all $\varepsilon > 0$ we have
	\begin{equation}
		\label{equation:definition:completeConvergence}
		\sum_{n = 1}^{\infty}{\mathbb{P}\left[|X_n - C| > \varepsilon\right]} < \infty \ \mbox{.}
	\end{equation}
\end{definition}

\ifdefined\FULLVERSION

\begin{theorem}[basic limit theorem for Euclidean functionals -- \cite{ProbabilityTheoryOfClassicalEuclideanOptimizationProblems}, originally \cite{AMatchingProblemAndSubadditiveEuclideanFunctionals}]
	\label{theorem:basicLimitTheoremForEuclideanFunctionals}
	Let $X_i$, $1 \leq i \leq n$, be independent and identically distributed random variables with values in the unit $\mbox{\it dim}$-dimensional rectangle $[0, 1]^{\mbox{\it dim}}$, $\mbox{\it dim} \geq 2$.
	
	If $P_B$ is a smooth superadditive Euclidean functional on $\rz^{\mbox{\it dim}}$, $\mbox{\it dim} \geq 2$, then
	\begin{equation}
		\label{equation:theorem:basicLimitTheorem:equation1}
		\lim_{n \to \infty}{\frac{P_B(X_1, X_2, \dots, X_n)}{n^{\frac{\mbox{\it dim} - 1}{\mbox{\it dim}}}}} = \alpha(P_B, \mbox{\it dim}) \quad \mbox{c.c.} \ \text{,}
	\end{equation}
	where $\alpha(P_B, \mbox{\it dim})$ is a positive constant.
	
	If $P$ is an Euclidean functional on $\rz^{\mbox{\it dim}}$, $\mbox{\it dim} \geq 2$, which is pointwise close to $P_B$, then
	\begin{equation}
		\label{equation:theorem:basicLimitTheorem:equation2}
		\lim_{n \to \infty}{\frac{P(X_1, X_2, \dots, X_n)}{n^{\frac{\mbox{\it dim} - 1}{\mbox{\it dim}}}}} = \alpha(P_B, \mbox{\it dim}) \quad \mbox{c.c.}
	\end{equation}
	\begin{proof}
		See~\cite{ProbabilityTheoryOfClassicalEuclideanOptimizationProblems}.
	\end{proof}
\end{theorem}

We can prove our result now.

\begin{lemma}
	\label{lemma:2MatchingInRandomGeometricGraphsCompleteConvergence}
	The 2-matching functional $L$ and the boundary 2-matching functional $L_B$ fulfill the conditions of Theorem~\ref{theorem:basicLimitTheoremForEuclideanFunctionals}.
	\begin{proof}
		The proof is a modification of similar proofs for other combinatorial optimization problems contained in~\cite{ProbabilityTheoryOfClassicalEuclideanOptimizationProblems}.
		
		\begin{enumerate}[label={(\arabic*)}]
			\item First, we show that the boundary 2-matching functional $L_B$ is a superadditive Euclidean functional. Equalities~(\ref{equation:definition:subadditiveAndSuperadditiveEuclideanFunctional:condition1}), (\ref{equation:definition:subadditiveAndSuperadditiveEuclideanFunctional:condition2}) and~(\ref{equation:definition:subadditiveAndSuperadditiveEuclideanFunctional:condition3}) are fulfilled obviously.
			
				Let us now show the superadditivity. We can distinguish $2$ cases in general:
				\begin{enumerate}
					\item Either the solution over the whole rectangle $R$ does not cross the boundary between the rectangles $R_1$ and $R_2$ at all or
					\item at least one subtour crosses the boundary between the rectangles $R_1$ and $R_2$.
				\end{enumerate}
				
				$L_B(F, R) = L_B(F, R_1) + L_B(F, R_2)$ obviously holds in the first case.
				
				Let us now consider a subtour crossing the boundary between the rectangles $R_1$ and $R_2$ (for an example see Figure~\ref{figure:theorem:basicLimitTheore:proof:exampleIllustratingTheSuperadditivityOfTheBoundary2MatchingFunctionalLB}). W.l.o.g. we can assume that the boundary is crossed between the points $v_1$ and $v_2$, and $v_3$ and $v_4$ and that the crossing points are $x$ and $y$ respectively. Furthermore, w.l.o.g. we can assume that $v_1, v_3 \in R_1$. Then the new subtour, containing the vertices $v_1$ and $v_3$, and lying in the rectangle $R_1$, consists of the following parts:
				\begin{itemize}
					\item the same path between the vertices $v_1$ and $v_3$ belonging to the rectangle $R_1$ as in the whole rectangle $R$,
					\item the orthogonal connections between the vertices $v_1$ and $v_3$ and the boundary between the both rectangles, and finally
					\item a piece of this boundary (see also Figure~\ref{figure:theorem:basicLimitTheore:proof:exampleIllustratingTheSuperadditivityOfTheBoundary2MatchingFunctionalLB}).
				\end{itemize}
				
				We have to choose the vertices $a$ and $b$ on this boundary in such a way that $a = \argmin{\alpha \in \partial R_1 \cap \partial R_2}{\{d(v_1, \alpha)\}}$ and $b = \argmin{\beta \in \partial R_1 \cap \partial R_2}{\{d(v_3, \beta)\}}$ hold in order to achieve the minimality. Due to this choice of the vertices $a$ and $b$ we can write $d(v_1, a) \leq d(v_1, x)$ and $d(v_3, b) \leq d(v_3, y)$ and since $d(a, b) = 0$ and the remaining part of the subtour belonging to the rectangle $R_1$ yield the same contribution to the objective value, we can claim that the contribution of this new subtour to the objective value is smaller or equal to the contribution of the part of the original subtour lying in the rectangle $R_1$. The same argument can be used for the second rectangle $R_2$ and for all other subtours crossing the boundary between the rectangles $R_1$ and $R_2$.
			
				\begin{figure}[htbp]
					\centering
					\begin{tikzpicture}[
							scale=0.6,
							node/.style={circle, draw=black!100, fill=white!0, thick, inner sep=0pt, minimum size=5mm}
						]
						
						\draw (0.00, 0.00) -- (20.00, 0.00) -- (20.00, 10.00) -- (0.00, 10.00) -- cycle;
						\draw (10.00, 0.00) -- (10.00, 10.00);
						
						\node[circle, draw=black!100, fill=black!100, thick, inner sep=0pt, minimum size=0.5mm, label=above:{\color{black} $v_1$}] (node0) at (8.50, 9.00) {};
						\node[circle, draw=black!100, fill=black!100, thick, inner sep=0pt, minimum size=0.5mm, label=above:{\color{black} $v_2$}] (node1) at (11.50, 7.00) {};
						\node[circle, draw=black!100, fill=black!100, thick, inner sep=0pt, minimum size=0.5mm, label=below:{\color{black} $v_3$}] (node2) at (8.50, 3.00) {};
						\node[circle, draw=black!100, fill=black!100, thick, inner sep=0pt, minimum size=0.5mm, label=below:{\color{black} $v_4$}] (node3) at (11.50, 1.00) {};
						\draw (node0) to (node1) to (node3) to (node2) to (node0);

							\node[circle, draw=black!100, fill=black!100, thick, inner sep=0pt, minimum size=0.5mm, label=below:{\color{black} $v_5$}] (node4) at (1.00, 1.00) {};
						\node[circle, draw=black!100, fill=black!100, thick, inner sep=0pt, minimum size=0.5mm, label=right:{\color{black} $v_6$}] (node5) at (2.00, 2.00) {};
						\node[circle, draw=black!100, fill=black!100, thick, inner sep=0pt, minimum size=0.5mm, label=above:{\color{black} $v_7$}] (node6) at (1.50, 3.00) {};
						\draw (node4) to (node5) to (node6) to (node4);
						
						\node[circle, draw=black!100, fill=black!100, thick, inner sep=0pt, minimum size=0.5mm, label=below:{\color{black} $v_8$}] (node7) at (16.00, 5.00) {};
						\node[circle, draw=black!100, fill=black!100, thick, inner sep=0pt, minimum size=0.5mm, label=below:{\color{black} $v_9$}] (node8) at (17.50, 5.50) {};
						\node[circle, draw=black!100, fill=black!100, thick, inner sep=0pt, minimum size=0.5mm, label=above:{\color{black} $v_{10}$}] (node9) at (18.00, 7.00) {};
						\node[circle, draw=black!100, fill=black!100, thick, inner sep=0pt, minimum size=0.5mm, label=above:{\color{black} $v_{11}$}] (node10) at (16.50, 6.50) {};
						\draw (node7) to (node8) to (node9) to (node10) to (node7);
						
						\node[circle, draw=black!100, fill=black!100, thick, inner sep=0pt, minimum size=0.5mm, label=right:{\color{black} $x$}] (nodeX) at (10.00, 8.00) {};
						\node[circle, draw=black!100, fill=black!100, thick, inner sep=0pt, minimum size=0.5mm, label= left:{\color{black} $y$}] (nodeY) at (10.00, 2.00) {};

							\node[circle, draw=black!100, fill=black!100, thick, inner sep=0pt, minimum size=0.5mm, label=right:{\color{black} $a$}] (nodeA) at (9.98, 9.00) {};
						\node[circle, draw=black!100, fill=black!100, thick, inner sep=0pt, minimum size=0.5mm, label=right:{\color{black} $b$}] (nodeB) at (9.98, 3.00) {};
						\node[circle, draw=black!100, fill=black!100, thick, inner sep=0pt, minimum size=0.5mm, label=left:{\color{black} $a^\prime$}] (nodeAPrime) at (10.02, 7.00) {};
						\node[circle, draw=black!100, fill=black!100, thick, inner sep=0pt, minimum size=0.5mm, label=left:{\color{black} $b^\prime$}] (nodeBPrime) at (10.02, 1.00) {};
						\draw[thick] (node0) to (nodeA) to (nodeB) to (node2) to (node0);
						\draw[dashed, very thick] (node1) to (nodeAPrime) to (nodeBPrime) to (node3) to (node1);
						
						\node (nodR1) at (5.00, 5.00) {$R_1$};
						\node (nodR2) at (15.00, 5.00) {$R_2$};
					\end{tikzpicture}
					\caption{Example illustrating the superadditivity of the boundary 2-matching functional $L_B$.}
					\label{figure:theorem:basicLimitTheore:proof:exampleIllustratingTheSuperadditivityOfTheBoundary2MatchingFunctionalLB}
				\end{figure}
			\item Next, we check some properties of the 2-matching functional $L$. Equalities~(\ref{equation:definition:subadditiveAndSuperadditiveEuclideanFunctional:condition1}), (\ref{equation:definition:subadditiveAndSuperadditiveEuclideanFunctional:condition2}) and~(\ref{equation:definition:subadditiveAndSuperadditiveEuclideanFunctional:condition3}) obviously hold.
				
				Further, it is easy to see that the 2-matching functional $L$ fulfill the geometric subadditivity. Since we minimize, the minimum weighted 2-matching over the whole rectangle $R$ can have only a smaller objective value than the sum of the objective values for the rectangles $R_1$ and $R_2$ taken separately.
				
				As a next step, we have to prove the pointwise closeness of the 2-matching functional $L$ to the boundary 2-matching functional $L_B$.
				
				First, note that $L_B\left(F, [0, 1]^{\mbox{\it dim}}\right) \leq L\left(F, [0, 1]^{\mbox{\it dim}}\right)$ always hold (see~(\ref{equation:definition:2MatchingFunctionalAndBoundary2MatchingFunctional:boundary2MatchingFunctional})). Thus it suffices to show
				\begin{equation}
					\label{equation:lemma:2MatchingInRandomGeometricGraphsCompleteConvergence:proof:part1:equation1}
					L\left(F, [0, 1]^{\mbox{\it dim}}\right) \leq L_B\left(F, [0, 1]^{\mbox{\it dim}}\right) + C_7|F|^\frac{d - 1}{d}
				\end{equation}
				where $C_7 \defeq C_7(\mbox{\it dim})$ is a finite constant.
				
				Let $F^* \subseteq F$ be the set of vertices which are connected with the boundary $\partial [0, 1]^{\mbox{\it dim}}$ by a path. Now, we remove all edges incident with the vertices contained in the vertex set $F^*$. If $|F^*| < 3$, we can just put these vertices to an arbitrary subtour (if such a subtour exists) and get the above inequality (the increase of the objective value can be easily bounded by $4 \sqrt{\mbox{dim}}$ in this case). If $|F^*| \geq 3$, we can construct an minimum weighted 2-matching on this vertex set and obtain
				\begin{equation}
					\label{equation:lemma:2MatchingInRandomGeometricGraphsCompleteConvergence:proof:part1:equation2}
					L\left(F, [0, 1]^{\mbox{\it dim}}\right) \leq L_B\left(F, [0, 1]^{\mbox{\it dim}}\right) + L\left(F^*, [0, 1]^{\mbox{\it dim}}\right) \ \text{.}
				\end{equation}
				And since $|F^*| \leq |F|$, we can use Lemma~\ref{lemma:growthBounds} and get inequality~(\ref{equation:lemma:2MatchingInRandomGeometricGraphsCompleteConvergence:proof:part1:equation1}).
			\item Finally, we prove the smoothness of the boundary 2-matching functional $L_B$. We will show the simple and geometric subadditivity of this functional first in order to be ably to show the smoothness.
			
			Let $F$ and $G$ be finite sets in $[0, t]^{\mbox{\it dim}}$. If the minimum weighted 2-matching for the vertex set $F \cup G$ equals to the minimum weighted 2-matchings for the vertex sets $F$ and $G$ joined together, inequality~(\ref{equation:definition:simpleSubadditivityGeometricSubadditivityAndSuperAdditivity:simpleSubadditivity}) holds with equality. Since we can always construct such a solution for the vertex set $F \cup G$, the objective value can be only smaller in the other case (note that we minimize it).
				
				Let us now prove the geometric subadditivity in order to fulfill the conditions of Lemma~\ref{lemma:growthBounds}. We know that the 2-matching functional $L$ is geometric subadditive. Now, it is easy to see that the proof of inequality~(\ref{equation:lemma:2MatchingInRandomGeometricGraphsCompleteConvergence:proof:part1:equation1}) can be easily modified in order to obtain
				\begin{equation}
					\label{equation:lemma:2MatchingInRandomGeometricGraphsCompleteConvergence:proof:part2:equation0}
					L(F, R) \leq L_B(F, R) + C_7 \mbox{\it diam}(R) |F|^\frac{d - 1}{d}
				\end{equation}
				for an arbitrary $\mbox{\it dim}$-dimensional rectangle. Since $L_B(F, R) \leq L(F, R)$, we obtain inequality~(\ref{equation:definition:simpleSubadditivityGeometricSubadditivityAndSuperAdditivity:geometricSubadditivity}) immediately.
				
				Using the simple subadditivity and Lemma~\ref{lemma:growthBounds} we get for all finite sets $F, G \subset [0, 1]^{\mbox{\it dim}}$
				\begin{equation}
					\label{equation:lemma:2MatchingInRandomGeometricGraphsCompleteConvergence:proof:part2:equation1}
					\begin{split}
						L_B(F \cup G)
							& \leq L_B(F) + \left(C_1 + C_4 \sqrt{\mbox{\it dim}}\right) |G|^{\frac{\mbox{\it dim} - 1}{\mbox{\it dim}}}\\
							& \leq L_B(F) + C_5 |G|^{\frac{\mbox{\it dim} - 1}{\mbox{\it dim}}}\\
					\end{split}
				\end{equation}
				where $C_5 \defeq C_5(\mbox{\it dim})$ denotes a finite constant. This completes this part of the proof if $L_B(F \cup G) - L_B(F) \geq 0$. Hence we just need to show the following inequality
				\begin{equation}
					\label{equation:lemma:2MatchingInRandomGeometricGraphsCompleteConvergence:proof:part2:equation2}
					L_B(F \cup G) \geq L_B(F) - C_6 |G|^{\frac{\mbox{\it dim} - 1}{\mbox{\it dim}}}
				\end{equation}
				for some finite constant $C_6 \defeq C_6(\mbox{\it dim})$.
				
				Consider the global minimum weighted 2-matching on the vertex set $G \cup F$ and remove all edges from all subtours incident with a vertex $g \in G$. This yield at most $|G|$ paths of a length of at least $1$ containing only vertices from the vertex set $F$ and some isolated points $F^\prime \subseteq F$. Let $F^*$ denote the set of all endpoints of those paths. Clearly, we have $|F^\prime| \leq |G|$ and $|F^*| \leq 2 |G|$. Consider now the boundary matching functional $M_B(F^*)$ and the corresponding matching $m$. This matching together with the disconnected paths and with parts of the boundary of the $\mbox{\it dim}$-dimensional rectangle $[0, 1]^{\mbox{\it dim}}$ yield a set of subtours $\{\tilde{F_i}\}_{i = 1}^N$ for some particular positive integer $N$. Furthermore, we can construct an minimum weighted 2-matching on the vertex set $F^\prime$ and get a feasible minimum weighted 2-matching. We can write
				\begin{equation}
					\label{equation:lemma:2MatchingInRandomGeometricGraphsCompleteConvergence:proof:part2:equation3}
					L_B(F) \leq L_B(F \cup G) + M_B(F^*) + L_B(F^\prime) \ \text{.}
				\end{equation}
				By using Lemmas~\ref{lemma:theMatchingFunctionalMAndTheBoundaryMatchingFunctionalMBAreSubadditiveEuclideanFunctionals} and~\ref{lemma:growthBounds} we obtain
				\begin{equation}
					\label{equation:lemma:2MatchingInRandomGeometricGraphsCompleteConvergence:proof:part2:equation4}
					L_B(F) \leq L_B(F \cup G) + C_6\left(|F^*|^{\frac{\mbox{\it dim} - 1}{\mbox{\it dim}}} + |F^\prime|^{\frac{\mbox{\it dim} - 1}{\mbox{\it dim}}}\right) \ \text{.}
				\end{equation}
				And since $|F^\prime| \leq |G| \leq 2 |G|$ and $|F^*| \leq 2 |G|$, we get
				\begin{equation}
					\label{equation:lemma:2MatchingInRandomGeometricGraphsCompleteConvergence:proof:part2:equation5}
					\begin{split}
						L_B(F)	& \leq L_B(F \cup G) + C_6\left(\left(2 |G|\right)^{\frac{\mbox{\it dim} - 1}{\mbox{\it dim}}} + 2 \left(|G|\right)^{\frac{\mbox{\it dim} - 1}{\mbox{\it dim}}}\right)\\
								& \leq L_B(F \cup G) + C_6\left(|G|\right)^{\frac{\mbox{\it dim} - 1}{\mbox{\it dim}}} \ \text{.}\\
					\end{split}
				\end{equation}
				This is exactly inequality~(\ref{equation:lemma:2MatchingInRandomGeometricGraphsCompleteConvergence:proof:part2:equation2}).				
		\end{enumerate}
	\end{proof}
\end{lemma}

\fi

\begin{theorem}
	\label{theorem:SquareRootAsymptoticForThe2MatchingProblem}
	Let $G = (V, E)$ be a random Euclidean graph with $n = |V|$ vertices and let $d \colon E \to \mathbb{R}^+_0$ be the Euclidean distance function.
	Furthermore, let $M_2(G, d)$ be the length of an optimal $2$-matching. Then
	\begin{equation}
		\label{equation:theorem:SquareRootAsymptoticForThe2MatchingProblem:equation}
		\lim_{n \to \infty} \frac{M_2(G, d)}{\sqrt{n}} = \alpha\ \mbox{ c.c., where } \alpha > 0.
	\end{equation}
\end{theorem}
	\begin{proof}
		\ifdefined\FULLVERSION
			The theorem immediately follows from Theorem~\ref{theorem:basicLimitTheoremForEuclideanFunctionals} and Lemma~\ref{lemma:2MatchingInRandomGeometricGraphsCompleteConvergence}.
		\else
			See our accompanying technical report~\cite{GeneratingSubtourEliminationConstraintsForTheTSPFromPureIntegerSolutions}.
		\fi
	\end{proof}

Based on these results the following idea might lead to a proof that the expected cardinality
${\cal S^*}$ is polynomially bounded:
After the first iteration of the algorithm we have a solution possibly consisting of several
separate subtours of total asymptotic length $\alpha \sqrt{n} = \alpha_1 \sqrt{n}$.
If there are subtours, we add subtour elimination constraints (in fact at most $\lfloor\frac{n}{3}\rfloor$), resolve the enlarged ILP and get another solution whose asymptotic length is $\alpha_2 \sqrt{n}$. By proving that the expected length of the sequence
$\alpha = \alpha_1, \ldots, \alpha_{\#i} = \beta$ is polynomially bounded in $n$,
one would obtain that also $\mathbb{E}\left[\left|S^*\right|\right]$ is polynomially bounded
since only polynomially many subtours are added in each iteration.
Our intuition and computational tests illustrated in 
Figure~\ref{figure:meanNumberOfIterationsMeanLengthOfAnOptimalTSPTourAndMeanLengthOfAnOptimalWeighted2Matching}, upper graph, indicate that the length of this sequence could be proportional to $\sqrt{n}$ as well.
Unfortunately, we could not find the suitable techniques to show this step.

\medskip
A different approach is illustrated by Figure~\ref{figure:meanNumberOfSubtoursDuringTheMainApproachInRandomEuclideanInstances},
where we examine the mean number of subtours contained in every iteration.
In particular, we chose $n = 60$, generated $100000$ random Euclidean instances and sorted them by
the number of iterations \#i. required by \mainApproach.
The most frequent number of ILP solver runs was $7$ (dotted line), 
but we summarize the results for $5$ (full line),
$6$ (dashed), $8$ (loosely dashed) and $9$ (loosely dotted) necessary runs in this figure as well. For every iteration of every class (with respect to the number of involved ILP runs)
we compute the mean number of subtours contained in the respective solutions.
As can be expected these numbers of subtours are decreasing (in average)
over the number of iterations.
To allow a better comparison of this behavior for different numbers of iterations
we scaled the iteration numbers into the interval $[0,10]$
(horizontal axis of 
 Figure~\ref{figure:meanNumberOfSubtoursDuringTheMainApproachInRandomEuclideanInstances}).
It can be seen that the average number of subtours contained in an optimal $2$-matching
(first iteration)
is about $9.2$ while in the last iteration we trivially have only one tour.
Between these endpoints we can first observe a mostly convex behavior,
only in the last step before reaching the optimal TSP tour
a sudden drop occurs.
It would be interesting to derive an asymptotic description of these curves.
An intuitive guess would point to an exponential function,
but so far we could not find a theoretical justification of this claim.


\section{Conclusions}
	\label{section:conclusions}

In this paper we provide a  ``test of concept''
of a very simple approach to solve TSP instances of medium size to optimality
by exploiting the power of current ILP solvers.
The approach consists of iteratively solving ILP models with relaxed subtour elimination constraints
to integer optimality.
Then it is easy to find integral subtours and add the corresponding subtour elimination constraints to the ILP model.
Iterating this process until no more subtours are contained in the solution
obviously solves the TSP to optimality.

In this work we focus on the structure of subtour elimination constraints and how to find a ``good'' set of subtour elimination constraints in reasonable time.
Therefore, we aim to identify the local structure of the vertices of a given TSP instance
by running a clustering algorithm.
Based on empirical observations and results from random graph theory
we further extend this clustering-based approach and develop a hierarchical clustering
method with a mechanism to identify subtour elimination constraints as ``relevant'',
if they appear in consecutive iterations of the algorithm.

We mostly refrained from adding additional features which are highly likely to improve the performance considerably, such as starting heuristics (cf.\ Section~\ref{sec:compstart}),
lower bounds or adding additional cuts.
In the future it might be interesting to explore the limits of performance
one can reach with a purely integer linear programming approach by adding these improvements.
Clearly, we can not expect such a basic approach to compete with the performance of
{\em Concorde}~\cite{TheTravelingSalesmanProblemAComputionalStudy},
which has been developed over many years
and basically includes all theoretical and technical developments known so far.
However, it turns out that most of the standard benchmark instances with up to $400$ vertices
can be solved in a few minutes by this purely integer strategy.

Finally, we briefly discussed some theoretical aspect 
for random Euclidean graphs which could lead
to polyhedral results in the expected case.

\subsection*{Acknowledgements}

The research was funded by the Austrian Science Fund (FWF): P23829-N13.

\medskip
We would like to thank the developers of the SCIP MIP-solver from the Konrad-Zuse-Zentrum f\"{u}r Informationstechnik Berlin, especially Mr Gerald Gamrath, for their valuable support.

\ifdefined\FULLVERSION\else{}
\fi

\bibliographystyle{IEEEtranSN}
\bibliography{GeneratingSEC_UP}

\setlength{\LTcapwidth}{\textheight}
\begin{landscape}
	\section*{Appendix}
		\label{sec:appendix}
		
		\vspace*{-0.7cm}
		
		\begin{multicols}{3}
			\begin{figurehere}
				\begin{center}

				\end{center}
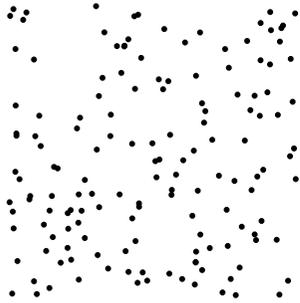
				\caption{Instance {\em RE\_A\_150}. Euclidean distances between vertices.}
				\label{figure:instance:REA150}
			\end{figurehere}
			
			$ $
			
			$ $
			
			$ $

			$ $
			
			\begin{figurehere}
				\begin{center}
					%
				\end{center}
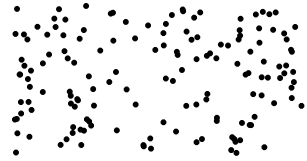
				\caption{Instance {\em kroB150}. Euclidean distances between vertices.}
				\label{figure:instance:kroB150}
			\end{figurehere}
			
			\begin{figurehere}
				\begin{center}
					%
				\end{center}
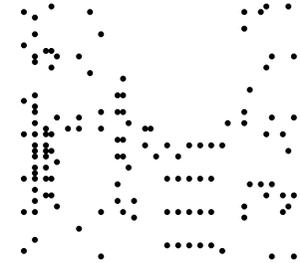
				\caption{Instance {\em u159}. Euclidean distances between vertices.}
				\label{figure:instance:u159}
			\end{figurehere}
		\end{multicols}
		
		\begin{multicols}{3}
			\begin{figurehere}
				\centering
				\begin{comment:instanceREA150MainIdeaOfOurApproachIterationAll}
					%
				\end{comment:instanceREA150MainIdeaOfOurApproachIterationAll}
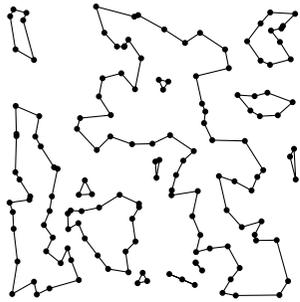
\caption{Instance {\em RE\_A\_150}: Main idea of our approach -- iteration~$0$.}
				\label{figure:instanceREA150MainIdeaOfOurApproachIteration0}
			\end{figurehere}

			\begin{figurehere}
				\centering
				\begin{comment:instanceREA150MainIdeaOfOurApproachIterationAll}
					%
				\end{comment:instanceREA150MainIdeaOfOurApproachIterationAll}
				\caption{
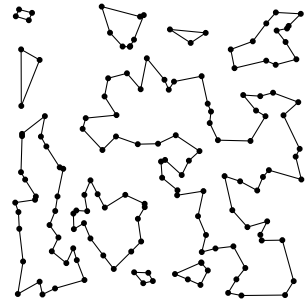
				{\em RE\_A\_150}: iteration~$1$.}
				\label{figure:instanceREA150MainIdeaOfOurApproachIteration1}
			\end{figurehere}
			
			\begin{figurehere}
				\centering
				\begin{comment:instanceREA150MainIdeaOfOurApproachIterationAll}
					%
				\end{comment:instanceREA150MainIdeaOfOurApproachIterationAll}
				\caption{
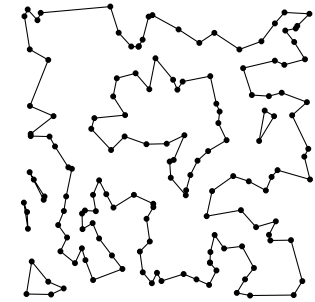
				{\em RE\_A\_150}: iteration~$2$.}
				\label{figure:instanceREA150MainIdeaOfOurApproachIteration2}
			\end{figurehere}
		\end{multicols}
		
		\newpage
		
		\begin{multicols}{3}
			\begin{figurehere}
				\centering
				\begin{comment:instanceREA150MainIdeaOfOurApproachIterationAll}
					%
				\end{comment:instanceREA150MainIdeaOfOurApproachIterationAll}
				\caption{
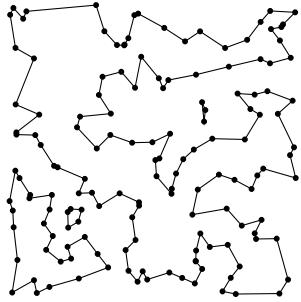
				{\em RE\_A\_150}: iteration~$3$.}
				\label{figure:instanceREA150MainIdeaOfOurApproachIteration3}
			\end{figurehere}
			
			\begin{figurehere}
				\centering
				\begin{comment:instanceREA150MainIdeaOfOurApproachIterationAll}
					%
				\end{comment:instanceREA150MainIdeaOfOurApproachIterationAll}
				\caption{
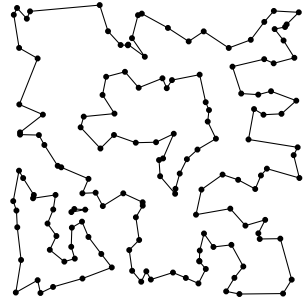
				{\em RE\_A\_150}: iteration~$4$.}
				\label{figure:instanceREA150MainIdeaOfOurApproachIteration4}
			\end{figurehere}

			\begin{figurehere}
				\centering
				\begin{comment:instanceREA150MainIdeaOfOurApproachIterationAll}
					%
				\end{comment:instanceREA150MainIdeaOfOurApproachIterationAll}
				\caption{
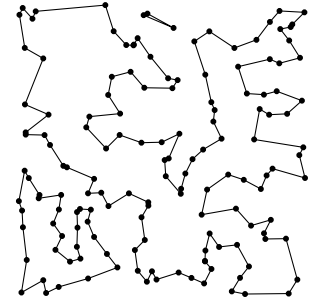
				{\em RE\_A\_150}: iteration~$5$.}
				\label{figure:instanceREA150MainIdeaOfOurApproachIteration5}
			\end{figurehere}
		\end{multicols}
		
		\begin{multicols}{3}
			\begin{figurehere}
				\centering
				\begin{comment:instanceREA150MainIdeaOfOurApproachIterationAll}
					%
				\end{comment:instanceREA150MainIdeaOfOurApproachIterationAll}
				\caption{
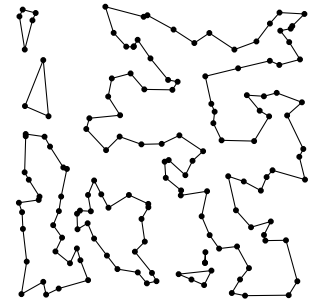
				{\em RE\_A\_150}: iteration~$8$.}
				\label{figure:instanceREA150MainIdeaOfOurApproachIteration8}
			\end{figurehere}
		\end{multicols}
		
		\newpage
		
		\begin{multicols}{3}
			\begin{figurehere}
				\centering
				\begin{comment:instanceREA150MainIdeaOfOurApproachIterationAll}
					%
				\end{comment:instanceREA150MainIdeaOfOurApproachIterationAll}
				\caption{
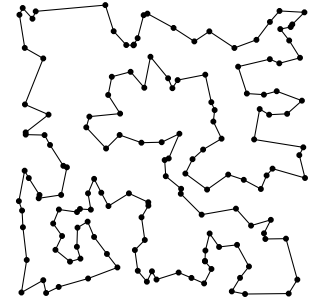
				{\em RE\_A\_150}: iteration~$11$.}
				\label{figure:instanceREA150MainIdeaOfOurApproachIteration11}
			\end{figurehere}
		\end{multicols}
		
		\begin{figure}[H]
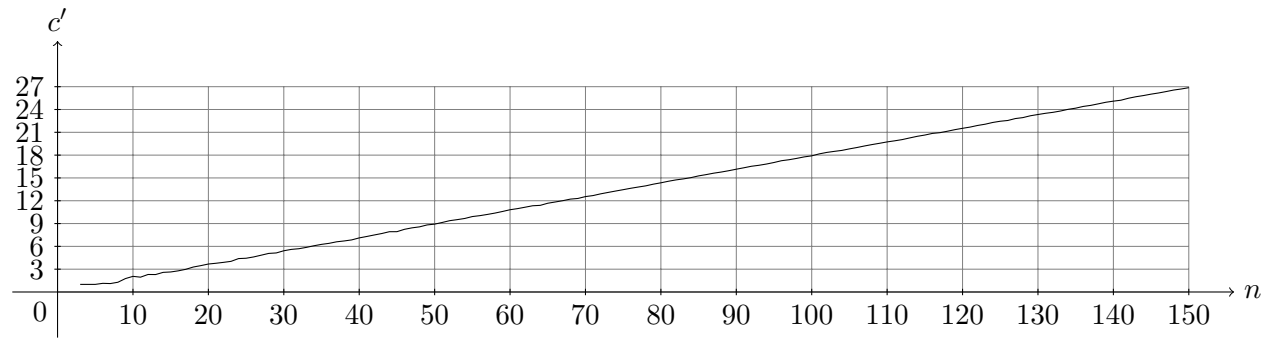

			\begin{center}
				%
			\end{center}
			\caption{Restricted clustering with $c = n$ on random Euclidean graphs with minimum cluster size 3. The number of  obtained clusters $c^\prime$ is plotted for every $n$.
For every number of vertices $n$ we created $100000$ graphs.}
			 \label{figure:progressOfTheNumberOfObtainedClustersCPrimeOnTheNumberOfVerticesNForTheRestrictedClusteringWhereCEqualNInRandomEuclideanGraphs}
		\end{figure}
		
		\begin{figure}[H]
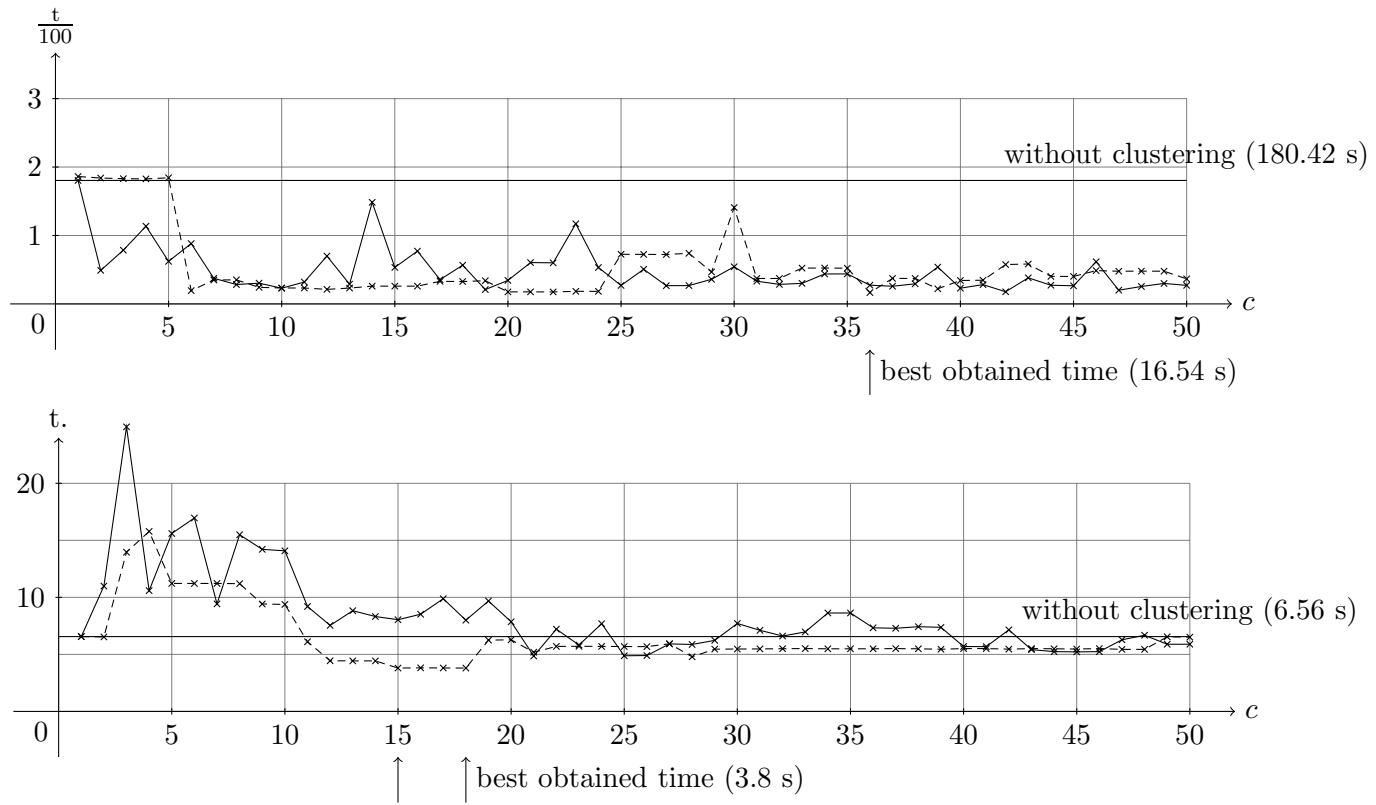

			\begin{center}
				%
			\end{center}
			\caption{Computation time $t$ in seconds depending on the number of clusters $c$ for
 clustering (full line) and for restricted clustering (dashed).
Illustrative  instances {\em kroB150} (upper figure) and {\em u159} (lower figure).}
			 \label{figure:progressOfTheAmountOfTheComputationTimeTInSecondsOnTheParameterC}
		\end{figure}
		
		\begin{figure}[H]
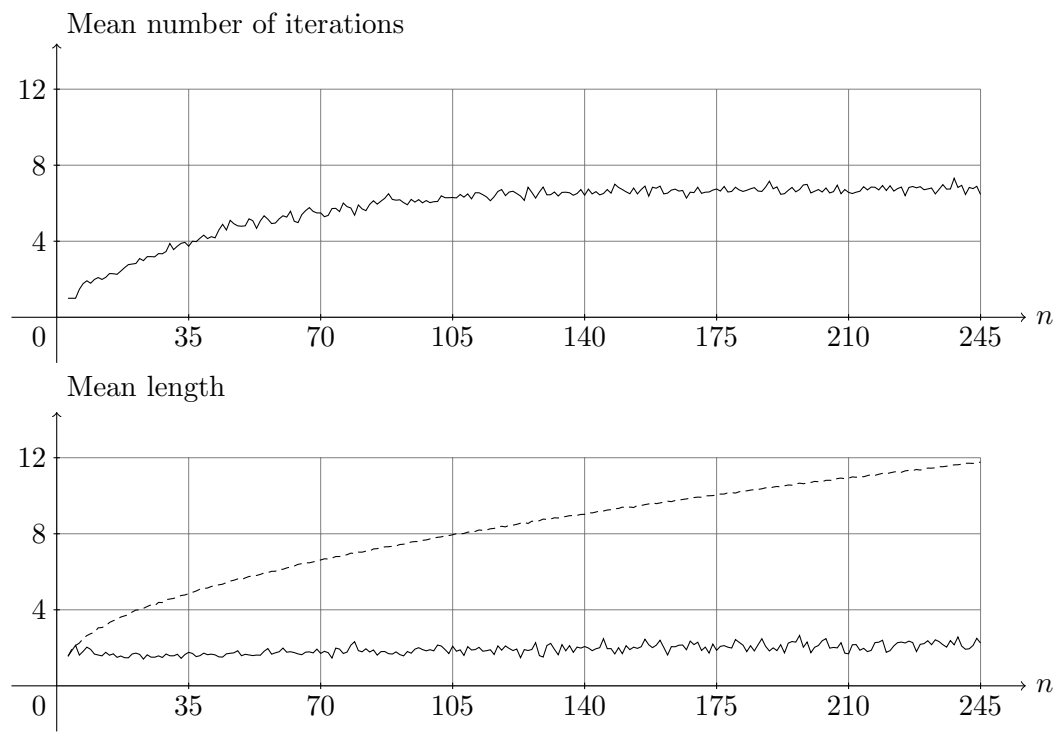

			\begin{center}
				\begin{comment:computationalResultsForRandomEuclideanInstances}
					%
				
				\end{comment:computationalResultsForRandomEuclideanInstances}
			\end{center}
			\caption{Mean number of iterations used by the \mainApproach\ (upper figure), mean length of an optimal TSP tour (lower figure, dashed) and mean length of an optimal weighted 2-matching (lower figure, full line) in random Euclidean graphs. For every number of vertices $n$ we created $100$ graphs.}
			 \label{figure:meanNumberOfIterationsMeanLengthOfAnOptimalTSPTourAndMeanLengthOfAnOptimalWeighted2Matching}
		\end{figure}
		
		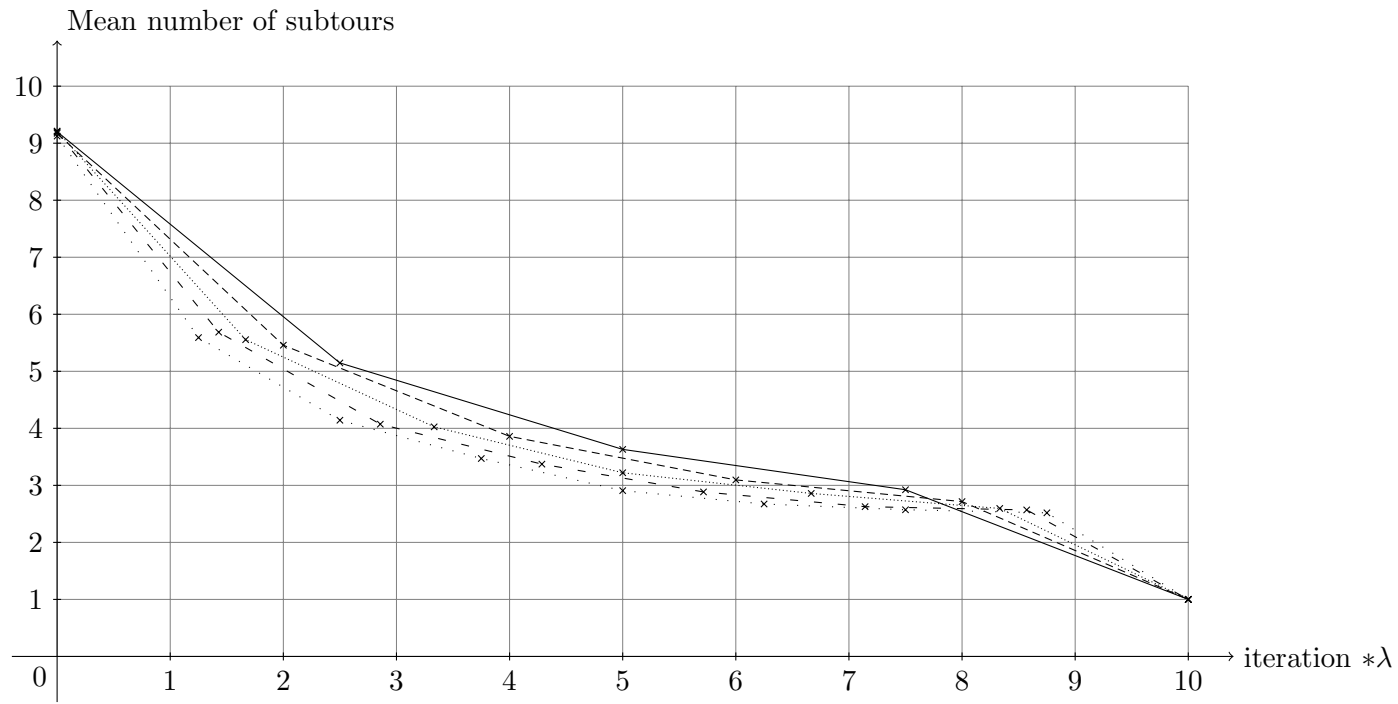
\begin{figure}[H]
			\begin{center}
				\begin{comment:computationalResultsForRandomEuclideanInstances}
					\begin{tikzpicture}[xscale=1.5, yscale=0.75]
						\draw[very thin, color=gray, ystep=1, xstep=1] (0, 0) grid (10, 10);
					
						\draw[black, shift={(0*10/4, 9.201923)}] (-0.8pt,-1.6pt) -- (0.8pt,1.6pt) (-0.8pt,1.6pt) -- (0.8pt,-1.6pt);
						\draw[black, shift={(1*10/4, 5.146635)}] (-0.8pt,-1.6pt) -- (0.8pt,1.6pt) (-0.8pt,1.6pt) -- (0.8pt,-1.6pt);
						\draw[black, shift={(2*10/4, 3.630769)}] (-0.8pt,-1.6pt) -- (0.8pt,1.6pt) (-0.8pt,1.6pt) -- (0.8pt,-1.6pt);
						\draw[black, shift={(3*10/4, 2.923077)}] (-0.8pt,-1.6pt) -- (0.8pt,1.6pt) (-0.8pt,1.6pt) -- (0.8pt,-1.6pt);
						\draw[black, shift={(4*10/4, 1.000000)}] (-0.8pt,-1.6pt) -- (0.8pt,1.6pt) (-0.8pt,1.6pt) -- (0.8pt,-1.6pt);
						\draw[black] (0*10/4, 9.201923) -- (1*10/4, 5.146635) -- (2*10/4, 3.630769) -- (3*10/4, 2.923077) -- (4*10/4, 1.000000);
						
						\draw[black, shift={(0*10/5, 9.173446)}] (-0.8pt,-1.6pt) -- (0.8pt,1.6pt) (-0.8pt,1.6pt) -- (0.8pt,-1.6pt);
						\draw[black, shift={(1*10/5, 5.457627)}] (-0.8pt,-1.6pt) -- (0.8pt,1.6pt) (-0.8pt,1.6pt) -- (0.8pt,-1.6pt);
						\draw[black, shift={(2*10/5, 3.859322)}] (-0.8pt,-1.6pt) -- (0.8pt,1.6pt) (-0.8pt,1.6pt) -- (0.8pt,-1.6pt);
						\draw[black, shift={(3*10/5, 3.095480)}] (-0.8pt,-1.6pt) -- (0.8pt,1.6pt) (-0.8pt,1.6pt) -- (0.8pt,-1.6pt);
						\draw[black, shift={(4*10/5, 2.715254)}] (-0.8pt,-1.6pt) -- (0.8pt,1.6pt) (-0.8pt,1.6pt) -- (0.8pt,-1.6pt);
						\draw[black, shift={(5*10/5, 1.000000)}] (-0.8pt,-1.6pt) -- (0.8pt,1.6pt) (-0.8pt,1.6pt) -- (0.8pt,-1.6pt);
						\draw[black, densely dashed] (0*10/5, 9.173446) -- (1*10/5, 5.457627) -- (2*10/5, 3.859322) -- (3*10/5, 3.095480) -- (4*10/5, 2.715254) -- (5*10/5, 1.000000);
						
						\draw[black, shift={(0*10/6, 9.212955)}] (-0.8pt,-1.6pt) -- (0.8pt,1.6pt) (-0.8pt,1.6pt) -- (0.8pt,-1.6pt);
						\draw[black, shift={(1*10/6, 5.554656)}] (-0.8pt,-1.6pt) -- (0.8pt,1.6pt) (-0.8pt,1.6pt) -- (0.8pt,-1.6pt);
						\draw[black, shift={(2*10/6, 4.025911)}] (-0.8pt,-1.6pt) -- (0.8pt,1.6pt) (-0.8pt,1.6pt) -- (0.8pt,-1.6pt);
						\draw[black, shift={(3*10/6, 3.221053)}] (-0.8pt,-1.6pt) -- (0.8pt,1.6pt) (-0.8pt,1.6pt) -- (0.8pt,-1.6pt);
						\draw[black, shift={(4*10/6, 2.858300)}] (-0.8pt,-1.6pt) -- (0.8pt,1.6pt) (-0.8pt,1.6pt) -- (0.8pt,-1.6pt);
						\draw[black, shift={(5*10/6, 2.595142)}] (-0.8pt,-1.6pt) -- (0.8pt,1.6pt) (-0.8pt,1.6pt) -- (0.8pt,-1.6pt);
						\draw[black, shift={(6*10/6, 1.000000)}] (-0.8pt,-1.6pt) -- (0.8pt,1.6pt) (-0.8pt,1.6pt) -- (0.8pt,-1.6pt);
						\draw[black, densely dotted] (0*10/6, 9.212955) -- (1*10/6, 5.554656) -- (2*10/6, 4.025911) -- (3*10/6, 3.221053) -- (4*10/6, 2.858300) -- (5*10/6, 2.595142) -- (6*10/6, 1.000000);
		
						\draw[black, shift={(0*10/7, 9.192053)}] (-0.8pt,-1.6pt) -- (0.8pt,1.6pt) (-0.8pt,1.6pt) -- (0.8pt,-1.6pt);
						\draw[black, shift={(1*10/7, 5.686093)}] (-0.8pt,-1.6pt) -- (0.8pt,1.6pt) (-0.8pt,1.6pt) -- (0.8pt,-1.6pt);
						\draw[black, shift={(2*10/7, 4.072848)}] (-0.8pt,-1.6pt) -- (0.8pt,1.6pt) (-0.8pt,1.6pt) -- (0.8pt,-1.6pt);
						\draw[black, shift={(3*10/7, 3.372185)}] (-0.8pt,-1.6pt) -- (0.8pt,1.6pt) (-0.8pt,1.6pt) -- (0.8pt,-1.6pt);
						\draw[black, shift={(4*10/7, 2.884768)}] (-0.8pt,-1.6pt) -- (0.8pt,1.6pt) (-0.8pt,1.6pt) -- (0.8pt,-1.6pt);
						\draw[black, shift={(5*10/7, 2.626490)}] (-0.8pt,-1.6pt) -- (0.8pt,1.6pt) (-0.8pt,1.6pt) -- (0.8pt,-1.6pt);
						\draw[black, shift={(6*10/7, 2.569536)}] (-0.8pt,-1.6pt) -- (0.8pt,1.6pt) (-0.8pt,1.6pt) -- (0.8pt,-1.6pt);
						\draw[black, shift={(7*10/7, 1.000000)}] (-0.8pt,-1.6pt) -- (0.8pt,1.6pt) (-0.8pt,1.6pt) -- (0.8pt,-1.6pt);
						\draw[black, loosely dashed] (0*10/7, 9.192053) -- (1*10/7, 5.686093) -- (2*10/7, 4.072848) -- (3*10/7, 3.372185) -- (4*10/7, 2.884768) -- (5*10/7, 2.626490) -- (6*10/7, 2.569536) -- (7*10/7, 1.000000);
						
						\draw[black, shift={(0*10/8, 9.124731)}] (-0.8pt,-1.6pt) -- (0.8pt,1.6pt) (-0.8pt,1.6pt) -- (0.8pt,-1.6pt);
						\draw[black, shift={(1*10/8, 5.591398)}] (-0.8pt,-1.6pt) -- (0.8pt,1.6pt) (-0.8pt,1.6pt) -- (0.8pt,-1.6pt);
						\draw[black, shift={(2*10/8, 4.141935)}] (-0.8pt,-1.6pt) -- (0.8pt,1.6pt) (-0.8pt,1.6pt) -- (0.8pt,-1.6pt);
						\draw[black, shift={(3*10/8, 3.470968)}] (-0.8pt,-1.6pt) -- (0.8pt,1.6pt) (-0.8pt,1.6pt) -- (0.8pt,-1.6pt);
						\draw[black, shift={(4*10/8, 2.907527)}] (-0.8pt,-1.6pt) -- (0.8pt,1.6pt) (-0.8pt,1.6pt) -- (0.8pt,-1.6pt);
						\draw[black, shift={(5*10/8, 2.673118)}] (-0.8pt,-1.6pt) -- (0.8pt,1.6pt) (-0.8pt,1.6pt) -- (0.8pt,-1.6pt);
						\draw[black, shift={(6*10/8, 2.572043)}] (-0.8pt,-1.6pt) -- (0.8pt,1.6pt) (-0.8pt,1.6pt) -- (0.8pt,-1.6pt);
						\draw[black, shift={(7*10/8, 2.518280)}] (-0.8pt,-1.6pt) -- (0.8pt,1.6pt) (-0.8pt,1.6pt) -- (0.8pt,-1.6pt);
						\draw[black, shift={(8*10/8, 1.000000)}] (-0.8pt,-1.6pt) -- (0.8pt,1.6pt) (-0.8pt,1.6pt) -- (0.8pt,-1.6pt);
						\draw[black, loosely dotted] (0*10/8, 9.124731) -- (1*10/8, 5.591398) -- (2*10/8, 4.141935) -- (3*10/8, 3.470968) -- (4*10/8, 2.907527) -- (5*10/8, 2.673118) -- (6*10/8, 2.572043) -- (7*10/8, 2.518280) -- (8*10/8, 1.000000);
										
						\draw[->] (-0.4, 0) -- (10.4, 0) node[right] {iteration $* \lambda$};
						\draw[->] (0, -0.8) -- (0, 10.8) node[above right] {Mean number of subtours};
						
						\foreach \pos in {0} \draw[shift={(\pos,0)}] node[below left] {$\pos$};
						\foreach \pos in {1, 2, 3, 4, 5, 6, 7, 8, 9, 10} \draw[shift={(\pos, 0)}] (0pt, 2pt) -- (0pt, -2pt) node[below] {$\pos$};
						\foreach \pos in {1, 2, 3, 4, 5, 6, 7, 8, 9, 10} \draw[shift={(0,\pos)}] (1pt, 0pt) -- (-1pt, 0pt) node[left] {$\pos$};
					\end{tikzpicture}
				\end{comment:computationalResultsForRandomEuclideanInstances}
			\end{center}
			\caption{Mean number of subtours during the \mainApproach\ in random Euclidean graphs for $n = 60$ sorted according to the number of iterations ($\lambda = $ {\color{black} $4/10$} (full line), {\color{black} $5/10$} (dashed), {\color{black} $6/10$} (dotted), {\color{black} $7/10$} (loosely dashed), {\color{black} $8/10$} (loosely dotted)). We created $100000$ graphs.}
			\label{figure:meanNumberOfSubtoursDuringTheMainApproachInRandomEuclideanInstances}
		\end{figure}
		
		\newpage

		\addtolength{\tabcolsep}{-4pt}
		
		\begin{center}
\begin{longtable}{l||rrr|rrr|rrr|rrr}
					{\bf instance} & \multicolumn{3}{c|}{\bf \mainApproach} & \multicolumn{3}{c|}{\bf \boldmath $HC \mid n$ \unboldmath} & \multicolumn{3}{c|}{\bf \boldmath $HC \mid 4 n / \log_2{n}$ \unboldmath} & \multicolumn{3}{c}{\bf \boldmath $HCD \mid 4 n / \log_2{n}$ \unboldmath}\\
					\vspace*{-2pt}
					& {\bf sec.} & {\bf \#i.} & {\bf \#c.} & {\bf sec.} & {\bf \#i.} & {\bf \#c.} & {\bf sec.} & {\bf \#i.} & {\bf \#c.} & {\bf sec.} & {\bf \#i.} & {\bf \#c.}\\
					\hline\hline
				\endfirsthead
					{\bf instance} & \multicolumn{3}{c|}{\bf \mainApproach} & \multicolumn{3}{c|}{\bf \boldmath $HC \mid n$ \unboldmath} & \multicolumn{3}{c|}{\bf \boldmath $HC \mid 4 n / \log_2{n}$ \unboldmath} & \multicolumn{3}{c}{\bf \boldmath $HCD \mid 4 n / \log_2{n}$ \unboldmath}\\
					\vspace*{-2pt}
					& {\bf sec.} & {\bf \#i.} & {\bf \#c.} & {\bf sec.} & {\bf \#i.} & {\bf \#c.} & {\bf sec.} & {\bf \#i.} & {\bf \#c.} & {\bf sec.} & {\bf \#i.} & {\bf \#c.}\\
					\hline\hline
				\endhead
				\endfoot
				\endlastfoot
				ch150
					& 13
						& 7
							& 74
								& 150
									& 2
										& 435
											& {\bf 9}
												& {\bf 5}
													& {\bf 223}
														& 14
															& 6
																& 129\\
				
				kroA150
					& 19
						& 7
							& 136
								& 11
									& 2
										& 268
											& {\bf 8}
												& {\bf 2}
													& {\bf 245}
														& 11
															& 4
																& 130\\
				
				kroB150
					& 179
						& 8
							& 148
								& 72
									& 2
										& 301
											& 34
												& 4
													& 315
														& {\bf 21}
															& {\bf 4}
																& {\bf 168}\\
				
				pr152
					& 16
						& 13
							& 184
								& {\bf 5}
									& {\bf 3}
										& {\bf 214}
											& {\bf 5}
												& {\bf 3}
													& {\bf 205}
														& 9
															& 4
																& 174\\
				
				u159
					& {\bf 6}
						& {\bf 4}
							& {\bf 49}
								& 29
									& 3
										& 292
											& 14
												& 4
													& 303
														& 11
															& 4
																& 140\\
				
				si175
					& 52
						& 10
							& 183
								& 99
									& 6
										& 494
											& {\bf 40}
												& {\bf 8}
													& {\bf 415}
														& 44
															& 7
																& 263\\
				
				brg180
					& 44
						& 4
							& 103
								& 54
									& 2
										& 185
											& 102
												& 18
													& 359
														& {\bf 24}
															& {\bf 2}
																& {\bf 27}\\
				
				rat195
					& 347
						& 6
							& 274
								& {\bf 241}
									& {\bf 3}
										& {\bf 491}
											& 272
												& 4
													& 419
														& 267
															& 5
																& 322\\
\end{longtable}
		\end{center}		
		
		\vspace*{-1.7cm}
		
		\begin{itemize}
			\item {\bf \mainApproach}
			\item {\bf \boldmath $HC \mid n$ \unboldmath} -- hierarchical clustering; the constraints cannot be dropped and the maximum size of a solved cluster is $u = n$ (i.e.\ in fact, there is no upper bound)
			\item {\bf \boldmath $HC \mid 4 n / \log_2{n}$ \unboldmath} -- hierarchical clustering; the constraints cannot be dropped and the maximum size of a solved cluster is $u = 4 \frac{n}{\log_2{n}}$
			\item {\bf \boldmath $HCD \mid 4 n / \log_2{n}$ \unboldmath} -- hierarchical clustering; the constraints can be dropped and the maximum size of a solved cluster is $u = 4 \frac{n}{\log_2{n}}$
		\end{itemize}

\newpage
		
		\begin{center}
\begin{longtable}{l||rrr|rrr|rrr|rrr|rrr}
					{\bf instance} & \multicolumn{3}{c|}{\bf \mainApproach} & \multicolumn{3}{c|}{\bf \boldmath $C \mid \lfloor n / 5 \rfloor$ \unboldmath} & \multicolumn{3}{c|}{\bf \boldmath $RC_3 \mid \lfloor n / 5 \rfloor$ \unboldmath} & \multicolumn{3}{c|}{\bf \boldmath $RC_3 \mid n$ \unboldmath} & \multicolumn{3}{c}{\bf \boldmath $HCD \mid 4 n / \log_2{n}$ \unboldmath}\\
					\vspace*{-2pt}
					& {\bf sec.} & {\bf \#i.} & {\bf \#c.} & {\bf sec.} & {\bf \#i.} & {\bf \#c.} & {\bf sec.} & {\bf \#i.} & {\bf \#c.} & {\bf sec.} & {\bf \#i.} & {\bf \#c.} & {\bf sec.} & {\bf \#i.} & {\bf \#c.}\\
					\hline\hline
				\endfirsthead
					{\bf instance} & \multicolumn{3}{c|}{\bf \mainApproach} & \multicolumn{3}{c|}{\bf \boldmath $C \mid \lfloor n / 5 \rfloor$ \unboldmath} & \multicolumn{3}{c|}{\bf \boldmath $RC_3 \mid \lfloor n / 5 \rfloor$ \unboldmath} & \multicolumn{3}{c|}{\bf \boldmath $RC_3 \mid n$ \unboldmath} & \multicolumn{3}{c}{\bf \boldmath $HCD \mid 4 n / \log_2{n}$ \unboldmath}\\
					\vspace*{-2pt}
					& {\bf sec.} & {\bf \#i.} & {\bf \#c.} & {\bf sec.} & {\bf \#i.} & {\bf \#c.} & {\bf sec.} & {\bf \#i.} & {\bf \#c.} & {\bf sec.} & {\bf \#i.} & {\bf \#c.} & {\bf sec.} & {\bf \#i.} & {\bf \#c.}\\
					\hline\hline
				\endhead
				\endfoot
				\endlastfoot
				ch150
					& 13
						& 7
							& 74
							
								& 12
									& 6
										& 114
										
											& {\bf 9}
												& {\bf 6}
													& {\bf 109}
													
														& 16
															& 7
																& 117
																
																	& 14
																		& 6
																			& 129\\
																			
				kroA150
					& 19
						& 7
							& 136
							
								& 25
									& 6
										& 187
										
											& 43
												& 6
													& 166
													
														& 33
															& 5
																& 185
																
																	& {\bf 11}
																		& {\bf 4}
																			& {\bf 130}\\
																			
				kroB150
					& 179
						& 8
							& 148
							
								& 53
									& 4
										& 219
										
											& 138
												& 5
													& 215
													
														& 44
															& 5
																& 202
																
																	& {\bf 21}
																		& {\bf 4}
																			& {\bf 168}\\
																			
				pr152
					& 16
						& 13
							& 184
							
								& 17
									& 12
										& 181
										
											& 17
												& 11
													& 204
													
														& 18
															& 12
																& 181
																
																	& {\bf 9}
																		& {\bf 4}
																			& {\bf 174}\\
																			
				u159
					& 6
						& 4
							& 49
							
								& 6
									& 4
										& 149
										
											& 5
												& 5
													& 151
													
														& {\bf 3}
															& {\bf 3}
																& {\bf 70}
																
																	& 11
																		& 4
																			& 140\\
\end{longtable}
		\end{center}
		
		\vspace*{-1.7cm}
		
		\begin{itemize}
			\item {\bf \mainApproach}
			\item {\bf \boldmath $C \mid \lfloor n / 5 \rfloor$ \unboldmath} -- clustering for $c = \lfloor\frac{n}{5}\rfloor$
			\item {\bf \boldmath $RC_3 \mid \lfloor n / 5 \rfloor$ \unboldmath} -- restricted clustering for $c = \lfloor\frac{n}{5}\rfloor$; the minimum size of a cluster is $3$
			\item {\bf \boldmath $RC_3 \mid n$ \unboldmath} -- restricted clustering for $c = n$; the minimum size of a cluster is $3$
			\item {\bf \boldmath $HCD \mid 4 n / \log_2{n}$ \unboldmath} -- hierarchical clustering; the constraints can be dropped and the maximum size of a solved cluster is $u = 4 \frac{n}{\log_2{n}}$
		\end{itemize}		

		\normalsize
\end{landscape}

\end{document}